\documentclass[a4paper, 12pt]{amsart}
\usepackage{a4wide} 
\usepackage[english]{babel} 
\usepackage[T1]{fontenc} 
\usepackage{amsfonts} 
\usepackage{amsmath} 
\usepackage{amssymb} 
\usepackage{amsthm} 
\usepackage{calrsfs} 
\usepackage{stmaryrd} 
\usepackage{graphicx} 
\usepackage{biblatex} 
\usepackage{csquotes} 
\usepackage{tikz} 

\DeclareMathAlphabet{\pazocal}{OMS}{zplm}{m}{n} 

\numberwithin{equation}{section}

\newtheorem{proposition}{Proposition}[section]
\newtheorem{lemma}{Lemma}[section]

\newtheorem{theorem}{Theorem}
\newtheorem{corollary}{Corollary}[section]

\newtheorem{remark}{Remark}[section]

\newcommand{\C}{\mathbb{C}} 
\newcommand{\R}{\mathbb{R}} 

\newcommand{\loc}{\mathrm{loc}} 
\newcommand{\comp}{\mathrm{comp}} 

\newcommand{\Ran}{\mathrm{Ran}} 
\newcommand{\rank}{\mathrm{rank}} 
\newcommand{\tr}{\mathrm{tr}} 
\newcommand{\Sp}{\mathrm{Sp}} 
\newcommand{\disc}{\mathrm{disc}} 
\newcommand{\Res}{\mathrm{Res}} 

\newcommand{\lp}{\left(} 
\newcommand{\rp}{\right)} 
\newcommand{\lb}{\left[} 
\newcommand{\rb}{\right]} 
\newcommand{\lcb}{\left\lbrace} 
\newcommand{\rcb}{\right\rbrace} 
\newcommand{\la}{\left\langle} 
\newcommand{\ra}{\right\rangle} 
\newcommand{\lv}{\left\vert} 
\newcommand{\rv}{\right\vert} 
\newcommand{\lV}{\left\Vert} 
\newcommand{\rV}{\right\Vert} 

\newcommand{\indic}{\mathbf{1}} 

\newcommand{\Dvec}{\mathbf{D}}

\newcommand{\Lcal}{\mathcal{L}}

\newcommand{\Kpaz}{\pazocal{K}}

\newcommand{\Mpaz}{\pazocal{M}}
\newcommand{\Npaz}{\pazocal{N}}

\newcommand{\Ppaz}{\pazocal{P}}

\newcommand{\Rpaz}{\pazocal{R}}

\newcommand{\Tpaz}{\pazocal{T}}

\newcommand{\Dfrak}{\mathfrak{D}}

\newcommand{\Tsf}{\mathsf{T}} 

\DeclareMathOperator{\Supp}{\mathrm{Supp}} 
\DeclareMathOperator{\tq}{\mathrm{ | }} 

\newcommand{\Dom}{\mathrm{Dom}} 
    
\title{Dirac resonances as non-self-adjoint eigenvalues}
\date{\today}
\author{Henry DUMANT}
\address{Univ. Bordeaux, CNRS, Bordeaux INP, IMB, UMR 5251, F-33400 Talence, France }
\email{henry.dumant@math.u-bordeaux.fr}
\begin{document}

\usetikzlibrary{patterns}.

\begin{abstract} 
    We consider resonances of (not necessarily self-adjoint) three-dimensional Dirac operators, defined as poles of the meromorphic continuation of the resolvent. We prove that a region of the resonance set can be characterized by the discrete spectrum of distorted Dirac operators, with preservation of multiplicities. As an application, we prove Kramer's degeneracy under time-reversal symmetry.
\end{abstract}

\maketitle

\section{Introduction}\label{Introduction}

\subsection*{Background.} Quantum resonances admit two complementary mathematical descriptions, each offering a different perspective. Defining resonances as poles of a suitable meromorphic continuation of the resolvent across the essential spectrum provides an intrinsic and global notion, independent of any auxiliary construction, but is not directly amenable to spectral methods. By contrast, realizing resonances as eigenvalues of non-self-adjoint operators makes available the full toolbox of spectral theory. In particular, for resonances lying close to the essential spectrum (the regime of greatest interest in many scattering problems) questions concerning multiplicity and perturbation reduce to the study of discrete eigenvalues. These two approaches were first developed for Schrödinger operators. The first approach originates in the work of Dolph, McLeod and Thoe \cite{DoMcLTh66} but was deeply generalized by Sjöstrand and Zworski \cite{SjZw91}. The second description originated with the method of complex scaling, whose central idea is to deform the configuration space. This approach was first introduced by Aguilar, Balslev and Combes \cite{AgCo71,BaCo71} with analytic dilations and later extended by Hunziker \cite{Hu86} to analytic distortions. The strategy is to first conjugate the operator and then analytically continue the deformation parameter. A more geometric construction was later developed by Sjöstrand and Zworski \cite{SjZw91}, who perform the deformation directly on the configuration space. In their paper, they establish the equivalence between resonances and the discrete spectrum of the deformed operator. Finally, an alternative realization based on escape functions and microlocal analysis was developed by Helffer and Sjöstrand \cite{HeSj86}, while its relation with complex scaling was clarified by Helffer and Martinez \cite{HeMa87}. In the Dirac setting, analytic dilations were introduced by Šeba \cite{Se88} and later generalized to analytic distortions by Khochman \cite{Kh07}, while the escape function approach was developed by Parisse \cite{Pa91}. Meromorphic continuation of the resolvent has also been established in several settings by Kungsmann, Melgard, Cheng and the author \cite{KuMe17,ChMe21,Du25}. Despite these parallel developments, the relation between the two main descriptions has apparently not been addressed in the Dirac setting. The present results will be used in our work in preparation \cite{Du26} to study resonances in the infinite mass limit.

\subsection*{Setting and goals.} In this paper we establish this missing correspondence for two scattering models associated with the Dirac operator in $\R^3$. The first one concerns a compactly supported and bounded matrix-valued potential (see Section \ref{Section3}), while the second deals with exterior domains endowed with MIT bag boundary conditions (see Section \ref{Section4}). They share a common geometric feature: the interaction is confined to a compact region. The free Dirac operator under study is $\beta-i\alpha\cdot\nabla$ ($\beta$ and $\alpha=(\alpha_1,\alpha_2,\alpha_3)$ being the Dirac matrices) whose spectrum is purely continuous and consists in the two branches $(-\infty,-1]\cup[1,\infty)$. Following \cite{SjZw91,Du25}, resonances are defined as poles of the meromorphic continuation of the resolvent (acting on weighted spaces) on the two-sheeted Riemann surface associated with the square root of $z^2-1$ (see Section \ref{Section3} and \ref{Section4} for precise definitions). In the spirit of \cite{Hu86,Kh07}, the distorted operators are obtained by conjugation with analytic deformations which coincide with the identity in a neighbourhood of the interaction region (see Section \ref{Section2} for the exact definition). The underlying mechanism is that the essential spectrum is moved in the complex plane, thereby uncovering resonances as isolated eigenvalues of the distorted operator. Theorems \ref{LinkResonanceEigenvalue} and \ref{MITLinkResonanceEigenvalue} can be stated in an informal way as follows.

\medskip

\textbf{Main result.} \textit{For both compactly supported and bounded perturbations of the free Dirac operator and exterior MIT bag models, a region of the resonance set is realized as the discrete spectrum of suitably distorted operators, with preservation of multiplicities.}

\medskip

This correspondence allows one to transport structural properties of the distorted operator to the resonance set. In particular, Kramers degeneracy is well known for time-reversal symmetric operators (see \cite[Section 10.4.5]{ReWo15}), and our main result allows us to extend this property to resonances (see  Corollaries \ref{KramersDegeneracy} and \ref{MITKramers} for precise statements). This was previously established by Parisse \cite{Pa91} for scalar smooth potentials in the context of resonances defined by escape functions.

\medskip

\textbf{Application.} \textit{Under time-reversal symmetry, these resonances have even multiplicity.}

\subsection*{Proof strategy and organization.} The proof follows the philosophy introduced by Dyatlov and Zworski \cite{DyZw19} for Schrödinger operators. Owing to the explicit structure of the free Dirac operator, a region free of the spectrum of the distorted free operator is obtained by constructing the distorted free resolvent directly by analytic continuation of the free resolvent kernel. We then introduce the compactly supported interaction and prove that localized resolvents of the physical and distorted operators coincide in a region of their common resolvent set. The proof relies almost entirely on elementary resolvent identities and analytic Fredholm theory. The resulting comparison of localized resolvents, together with the fact that localization preserves the ranks of the corresponding residues, yields the correspondence between resonances and discrete eigenvalues. The MIT bag case is obtained by adapting the same strategy to the presence of the boundary condition. Section \ref{Section2} is devoted to the distorted free Dirac operator. Sections \ref{Section3} and \ref{Section4} establish the main result for compactly supported potentials and for the MIT bag model, respectively. A distinctive feature of the argument is the systematic use of adjoints of the distorted operators, which play a crucial role both to establish that the free distorted operator has no discrete spectrum and to prove the evenness of the resonance multiplicities.

\section{The free distorted Dirac operator}\label{Section2}

In this section, we are interested in the free Dirac operator 
\begin{equation}\label{FreeDirac}
    \Dvec:=\beta-i\alpha\cdot\nabla.
\end{equation}
Let us recall that $\Dvec$ is self-adjoint in $L^2:=L^2(\R^3,\C^4)$ on its domain $H^1:=H^1(\R^3,\C^4)$ and that its spectrum is purely continuous equals to 
\begin{equation*}\label{FreeDiracSpectrum}
    \sigma_0:=(-\infty,-1]\cup[1,+\infty).
\end{equation*}

The distortion of $\psi\in L^2$ with parameter $\mu\in\R$ is defined as
\begin{equation*}\label{Distortion}
    (U_\mu\psi)(x):=\sqrt{\lv\det(I_3+\mu \nabla F(x))\rv}\psi(x+\mu F(x)).
\end{equation*}
where $F\in C^2(\R^3,\R^3)$ satisfies $F(x)=x$ for $\lv x\rv\gg 1$. In the next sections, $F$ will be chosen to vanish in a neighbourhood of the interaction region, but this assumption is not needed for the analysis of the free operators. The first step is to derive an explicit expression for the distorted gradient for real values of the distortion parameter and then extend this expression to complex values. This is the object of the following proposition. It is standard \cite{Hu86} and mainly computational. For completeness its proof is given in the appendix. We denote $L:=\lV \nabla F\rV_{L^\infty}$ and for $\mu\in\C$
\begin{equation*}\label{Diffeo}
    f_\mu:x\in\R^3\mapsto x+\mu F(x)\in\C^3.
\end{equation*}

\begin{proposition}\label{DistortedDerivatives}
    Let $\mu\in(-1/L,1/L)$. The operator $U_\mu$ is unitary on $L^2$ and preserves $H^1$. Moreover
    \begin{equation}\label{DistortedDerivativesFormula}
        U_\mu \partial_{x_j}U_\mu^{-1}=\sum_{k=1}^3\lb \nabla f_\mu(x)^{-1}\rb_{kj}\lp \partial_{x_k}-\frac{1}{2}\tr(\nabla f_\mu(x)^{-1}\partial_{x_k}\nabla f_\mu(x))\rp.
    \end{equation}
\end{proposition}

\noindent Let $\mu\in D(0,1/L):=\lcb z\in\C\mbox{, }\lv z\rv<1/L\rcb$. The r.h.s of \eqref{DistortedDerivativesFormula} still makes sense since $\nabla f_\mu(x)=I_3+\mu \nabla F(x)$ is invertible thanks to Neumann invertibility criterion. This motivates the following definition of the distorted gradient
\begin{equation*}\label{DistortedGradient}
    \left\{
    \begin{array}{ll}
        \nabla_\mu:=J_\mu(\nabla+a_\mu)\\
        
        [J_\mu(x)]_{jk}:=[\nabla f_\mu(x)^{-1}]_{kj}\\
        
        [a_\mu(x)]_k:=-\dfrac{1}{2}\tr(\nabla f_\mu(x)^{-1}\partial_{x_k}\nabla f_\mu(x)).
    \end{array}
    \right.
\end{equation*}

We define the distorted free Dirac operator by replacing $\nabla$ by $\nabla_\mu$ in \eqref{FreeDirac}
\begin{equation*}\label{DistortedFreeDirac}
    \Dvec_\mu:=\beta-i\alpha\cdot\nabla_\mu.
\end{equation*}
and now turn to its spectral analysis. In the case of dilations, i.e $F(x)=x$ for all $x\in\R^3$, then $\nabla_\mu=\frac{1}{1+\mu}\nabla$ and thanks to \cite{Se88} the spectrum of $\Dvec_\mu$ is
\begin{equation}\label{DistortedFreeSpectrum}
    \sigma_\mu:=\lcb z\in\C\mbox{, }(1+\mu)^2(z^2-1)\in[0,+\infty)\rcb.
\end{equation}
(see Figure \ref{FreeSpectralPicture}). 
\begin{figure}[htbp]
    \centering
    \begin{tikzpicture}[scale=1.5]
        \draw[->] (-3,0) -- (3,0) node[right] {$\Re z$};
        \draw[->] (0,-1.5) -- (0,1.5) node[above] {$\Im z$};
        
        \draw[very thick, blue] (-3,0) -- (-1,0);
        \draw[very thick, blue] (1,0) -- (3,0);
        
        \node[above, blue] at (2.2,0) {$\sigma_0$};
        
        \draw[thick,red,-]
          (1,0)
          .. controls (1,-0.8) and (1.8,-1.0)
          .. (3,-1.2);
        
        \draw[thick,red,-]
          (-1,0)
          .. controls (-1,0.8) and (-1.8,1.0)
          .. (-3,1.2);
        
        \node[red] at (2.4,-1.3) {$\sigma_\mu$};
        
        \fill (-1,0) circle (1.5pt);
        \fill (1,0) circle (1.5pt);
        
        \node[below] at (-1,0) {$-1$};
        \node[above] at (1,0) {$1$};
    \end{tikzpicture}
    \caption{Deformation of $\sigma_0$ as $\sigma_\mu$ when $\Im\mu>0$.}
    \label{FreeSpectralPicture}
\end{figure}
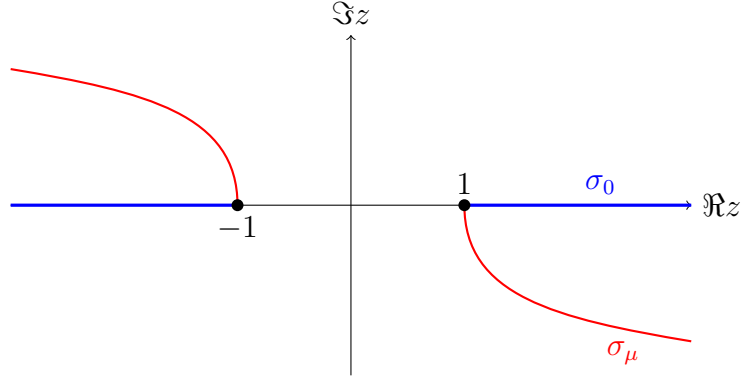
We go back to the case of distortions, i.e $F(x)=x$ only for $\lv x\rv\gg 1$. Thanks to \cite{Kh07} the essential spectrum of $\Dvec_\mu$ is $\sigma_\mu$ for $\lv\mu\rv\ll1$. The main goal of this section is to prove that $\Dvec_\mu$ has no spectrum in $\C\setminus\sigma_\mu$. We do not make use of the theory of analytic vectors but follow the approach given in \cite{DyZw19} in the case of Schrödinger operators. Precisely we construct an explicit candidate for $(\Dvec_\mu-z)^{-1}$ by analytic continuation of the kernel of $U_\mu(\Dvec-z)^{-1}U_\mu^{-1}$. In order to justify the resolvent construction carried out, the identification of the adjoint of $\Dvec_\mu$ will be needed. In contrast with the dilatated free Dirac operator, the distorted free Dirac operator involves variable complex coefficients, so both the expression and the domain of its adjoint require some care. Finally, although the applications considered in the literature only require small values of $\mu$, extending the constructions to the whole disk $D(0,1/L)$ requires relatively little additional work. This will be used later to get the even multiplicity result for a large class of resonances.

\subsection*{Adjoint of $\alpha\cdot\nabla_\mu$} The key ingredient in identifying $\Dom((\alpha\cdot\nabla_\mu)^*)$ is the following elliptic estimate. We provide an elementary proof which avoids the pseudo-differential machinery.

\begin{lemma}\label{Ellipticity}
    For all $\mu\in D(0,1/L)$ and $\psi\in H^1$
    \[\lV \nabla\psi\rV_{L^2}\lesssim\lV(\alpha\cdot\nabla_\mu)\psi\rV_{L^2}+\lV\psi\rV_{L^2}.\]
\end{lemma}

\begin{proof}
    For real $\mu$, since $U_\mu$ is unitary we get
    \begin{align*}
        \lV(\alpha\cdot\nabla_\mu)\psi\rV_{L^2}^2&=\lV(\alpha\cdot\nabla)U_\mu^{-1}\psi\rV_{L^2}^2\\
        &=\lV\nabla U_\mu^{-1}\psi\rV_{L^2}^2\\
        &=\lV\nabla_\mu\psi\rV_{L^2}^2.
    \end{align*}
    We conclude that for complex $\mu$
    \begin{equation}\label{alphaRemoved}
        \la(\alpha\cdot\nabla_{\overline{\mu}})\psi,(\alpha\cdot\nabla_\mu)\psi\ra_{L^2}=\la \nabla_{\overline{\mu}}\psi,\nabla_\mu\psi\ra_{L^2}.
    \end{equation}
    Indeed, the $L^2$-inner product being anti-linear in its first variable, linear in its second variable, $\overline{J_\mu}=J_{\overline{\mu}}$ and $\overline{a_\mu}=a_{\overline{\mu}}$ then both sides of \eqref{alphaRemoved} are analytic functions of $\mu$. Hence \eqref{alphaRemoved}, which holds for real $\mu$, extends to complex $\mu\in D(0,1/L)$ by analytic continuation. 
    
    The proof now relies on the following estimate, whose proof is postponed until the end of the present argument. There exists $c>0$ such that for all $x\in\R^3$ and $\xi\in\C^3$ 
    \begin{equation}\label{ACoercivity}
        \Re(J_{\overline{\mu}}(x)\xi)^*J_\mu(x)\xi+\frac{\Re\mu}{\Im\mu}\Im (J_{\overline{\mu}}(x)\xi)^* J_\mu(x)\xi\geq c\lv\xi\rv^2.
    \end{equation} 
    With $\psi=(\psi_1,\psi_2,\psi_3,\psi_4)$ we have 
    \begin{align*}
        &\la \nabla_{\overline{\mu}}\psi,\nabla_\mu\psi\ra_{L^2}\\
        &=\sum_{i=1}^4\la J_{\overline{\mu}}\nabla\psi_i,J_\mu \nabla\psi_i\ra_{L^2}+\la J_{\overline{\mu}}\nabla\psi_i,\psi_i J_\mu a_\mu\ra_{L^2}+\la\psi_iA_{\overline{\mu}} a_{\overline{\mu}},J_\mu \nabla\psi_i\ra_{L^2}+\la\psi_i J_{\overline{\mu}} a_{\overline{\mu}},\psi_i J_\mu a_\mu\ra_{L^2}.
    \end{align*}
    Using \eqref{alphaRemoved}, the latter expansion, \eqref{ACoercivity} and Cauchy-Schwarz's inequality we get for some $C>0$ independent of $\psi$
    \[\lV (\alpha\cdot\nabla_{\overline{\mu}})\psi\rV_{L^2}\lV (\alpha\cdot\nabla_\mu)\psi\rV_{L^2}\geq\sum_{i=1}^4c\lV\nabla\psi_i\rV_{L^2}^2-C\lV \nabla\psi_i\rV_{L^2}\lV\psi_i\rV_{L^2}-C\lV\psi_i\rV_{L^2}^2.\]
    Thanks to Young's inequality 
    \[C\lV \nabla\psi_i\rV_{L^2}\lV\psi_i\rV_{L^2}\leq \frac{c}{2}\lV \nabla\psi_i\rV_{L^2}^2+\frac{C^2}{2c}\lV\psi_i\rV_{L^2}^2\]
    therefore 
    \[\lV\nabla\psi\rV_{L^2}^2\lesssim\lV (\alpha\cdot\nabla_{\overline{\mu}})\psi\rV_{L^2}\lV (\alpha\cdot\nabla_\mu)\psi\rV_{L^2}+\lV\psi\rV_{L^2}^2.\]
    Eventually $\alpha\cdot\nabla_{\overline{\mu}}$ is bounded from $H^1$ to $L^2$ therefore the result is proved by absorption thanks to Young's inequality. 
    
    It remains to prove \eqref{ACoercivity}. The identities
    \begin{equation*}
        \left\{
        \begin{array}{ll}
            \Re(\nabla f_{\overline{\mu}}(x)\xi)^*\nabla f_\mu(x)\xi=\lv\xi\rv^2+2\Re\mu\Re\xi^*\nabla F(x)\xi+\lp(\Re\mu)^2-(\Im\mu)^2\rp\lv \nabla F(x)\xi\rv^2\\
            
            \Im(\nabla f_{\overline{\mu}}(x)\xi)^*\nabla f_\mu(x)\xi=2\Im\mu\Re\xi^*\nabla F(x)\xi+2\Re\mu\Im\mu\lv \nabla F(x)\xi\rv^2
        \end{array}
        \right.
    \end{equation*}
    provide
    \begin{align*}
        \Re(\nabla f_{\overline{\mu}}(x)\xi)^*\nabla f_\mu(x)\xi-\frac{\Re\mu}{\Im\mu}\Im(\nabla f_{\overline{\mu}}(x)\xi)^*\nabla f_\mu(x)\xi&=\lv\xi\rv^2-\lv\mu\rv^2\lv \nabla F(x)\xi\rv^2\\
        &\geq\lp1-\lv\mu\rv^2L^2\rp\lv\xi\rv^2.
    \end{align*}
    Changing $\xi$ by $\lp \nabla f_\mu(x)^\Tsf \nabla f_\mu(x)\rp^{-1}\xi$ and $\mu$ by $\overline{\mu}$ we deduce
    \begin{align*}
        \Re(J_{\overline{\mu}}(x)\xi)^*J_\mu(x)\xi+\frac{\Re\mu}{\Im\mu}\Im(J_{\overline{\mu}}(x)\xi)^*J_\mu(x)\xi&\geq\lp1-\lv\mu\rv^2L^2\rp\lv\lp \nabla f_\mu(x)^\Tsf \nabla f_\mu(x)\rp^{-1}\xi\rv^2\\
        &\geq\frac{1-\lv\mu\rv^2L^2}{\lV \nabla f_\mu(x)^\Tsf \nabla f_\mu(x)\rV^2}\lv\xi\rv^2.
    \end{align*}
    and \eqref{ACoercivity} follows since $\lV \nabla f_\mu(x)\rV\lesssim1$.
\end{proof}

We also need the following short computation.

\begin{lemma}
    Let $\chi\in C^1_c$ be real-valued and $\psi\in\Dom((\alpha\cdot\nabla_\mu)^*)$
    \begin{equation}\label{IPP}
        \int\chi(\alpha\cdot\nabla_\mu)^*\psi=\int\alpha\cdot(\nabla_{\overline{\mu}}\chi)\psi.
    \end{equation}
\end{lemma}

\begin{proof} Let $e_i$ be the $i-$th canonical vector of $\C^4$. The equality
    \begin{align*}
        \la\chi e_i,(\alpha\cdot\nabla_\mu)^*\psi\ra_{L^2}&=\la (\alpha\cdot\nabla_\mu)(\chi e_i),\psi\ra_{L^2}\\
        &=\la\alpha\cdot(\nabla_\mu\chi)e_i,\psi\ra_{L^2}
    \end{align*}
    provides the equality of the $i-$th components of each side of \eqref{IPP}.
\end{proof}

We are finally able to compute the adjoint of $\alpha\cdot\nabla_\mu$.

\begin{proposition}\label{FreeDistortedAdjoint}
    With domains $H^1$ one has $(\alpha\cdot\nabla_\mu)^*=-\alpha\cdot\nabla_{\overline{\mu}}$.
\end{proposition}

\begin{proof}
    Let $\psi,\phi\in H^1$. For real $\mu$
    \[\la\phi,(\alpha\cdot\nabla_\mu)\psi\ra_{L^2}=\la(-\alpha\cdot\nabla_\mu)\phi,\psi\ra_{L^2}.\]
    By analytic continuation we get for complex $\mu$
    \[\la\phi,(\alpha\cdot\nabla_\mu)\psi\ra_{L^2}=\la(-\alpha\cdot\nabla_{\overline{\mu}})\phi,\psi\ra_{L^2}.\]
    This shows that $H^1\subset\Dom((\alpha\cdot\nabla_\mu)^*)$ and $(\alpha\cdot\nabla_\mu)^*=-\alpha\cdot\nabla_{\overline{\mu}}$ on $H^1$. 
    
    Now let $\psi\in\Dom((\alpha\cdot\nabla_\mu)^*)$. It remains to prove that $\psi\in H^1$. We follow the strategy of Friedrich's Lemma (see \cite[Lemma 17.1.5]{Ho03}). We consider a mollifier $\chi_\varepsilon:=\frac{1}{\varepsilon^3}\chi\lp\frac{\cdot}{\varepsilon}\rp$. First
    \[(\alpha\cdot\nabla_{\overline{\mu}})(\chi_\varepsilon*\psi)(x)=\int\alpha\cdot[ J_{\overline{\mu}}(x)(\nabla\chi_\epsilon(x-y)+\chi_\varepsilon(x-y)a_{\overline{\mu}}(x))]\psi(y)dy\]
    moreover thanks to \eqref{IPP}
    \begin{align*}
        \chi_\varepsilon*(\alpha\cdot\nabla_\mu)^*\psi(x)&=\int\chi_\varepsilon(x-y)(\alpha\cdot\nabla_\mu)^*\psi(y)dy\\
        &=\int\alpha\cdot[ J_{\overline{\mu}}(y)(-\nabla\chi_\epsilon(x-y)+\chi_\varepsilon(x-y)a_{\overline{\mu}}(y))]\psi(y)dy
    \end{align*}
    therefore 
    \begin{align*}
        &(\alpha\cdot\nabla_{\overline{\mu}})(\chi_\varepsilon*\psi)(x)+\chi_\varepsilon*(\alpha\cdot\nabla_\mu)^*\psi(x)\\
        &=\int\alpha\cdot[(J_{\overline{\mu}}(x)-J_{\overline{\mu}}(y))\nabla\chi_\epsilon(x-y)+\chi_\varepsilon(x-y)(J_{\overline{\mu}}(x)a_{\overline{\mu}}(x)+J_{\overline{\mu}}(y)a_{\overline{\mu}}(y))]\psi(y)dy.
    \end{align*}
    Since for all $x,y\in\R^3$
    \[\lV J_{\overline{\mu}}(x)-J_{\overline{\mu}}(y)\rV\lesssim\lv x-y\rv \mbox{ and } \lV J_{\overline{\mu}}(x)a_{\overline{\mu}}(x)+J_{\overline{\mu}}(y)a_{\overline{\mu}}(y)\rV\lesssim 1\]
    we obtain
    \begin{align*}
        \lV(\alpha\cdot\nabla_{\overline{\mu}})(\chi_\varepsilon*\psi)+\chi_\varepsilon*(\alpha\cdot\nabla_\mu)^*\psi\rV_{L^2}&\lesssim \lV(\lv x\rv\lv \nabla\chi_\varepsilon\rv)*\lv\psi\rv\rV_{L^2}+\lV\chi_\varepsilon*\lv\psi\rv\rV_{L^2}\\
        &\lesssim(\lV\lv x\rv\lv \nabla\chi_\varepsilon\rv\rV_{L^1}+\lV\chi_\varepsilon\rV_{L^1})\lV\psi\rV_{L^2}\\
        &\lesssim(\lV\lv x\rv\lv \nabla\chi\rv\rV_{L^1}+\lV\chi\rV_{L^1})\lV\psi\rV_{L^2}
    \end{align*}
    where we used Young's estimate for the convolution to get the second estimate and the change of variable $y=x/\varepsilon$ to get the third one. Combining this with Lemma \ref{Ellipticity} we conclude that $(\chi_\varepsilon*\psi)_{\varepsilon>0}$ is bounded in $H^1$. Hence $(\chi_\varepsilon*\psi)_{\varepsilon>0}$ weakly converges in $H^1$ as $\varepsilon\rightarrow0$ (up to an extraction). Since $(\chi_\varepsilon*\psi)_{\varepsilon>0}$ converges to $\psi$ in $L^2$, we get $\psi\in H^1$. 
\end{proof}

\subsection*{Resolvent of $\Dvec_\mu$} For real $\mu$, $\Dvec_\mu$ and $\Dvec$ are conjugated by $U_\mu$ hence for $z\not\in\sigma_0$
\begin{equation}\label{ConjugateFreeResolvent}
    (\Dvec_\mu-z)^{-1}=U_\mu(\Dvec -z)^{-1}U_\mu^{-1}.
\end{equation}
Let us recall that $(\Dvec -z)^{-1}$ is the convolution operator with kernel
\begin{equation}\label{FreeKernel}
    K_z:x\mapsto\frac{e^{ i(z^2-1)^{1/2}\lv x\rv}}{4\pi\lv x\rv}\lp zI_4+\beta+\lp i+(z^2-1)^{1/2}\lv x\rv\rp\alpha\cdot\frac{ x}{\lv x\rv^2}\rp.
\end{equation}
where $\Im s^{1/2}>0$ if $s\in\C\setminus[0,+\infty)$. Using \eqref{ConjugateFreeResolvent}, \eqref{FreeKernel} and the change of variable formula we conclude that $(\Dvec_\mu-z)^{-1}$ is the integral operator with kernel 
\begin{equation}\label{FreeDistortedKernel}
    (x,y)\mapsto\sqrt{\det \nabla f_\mu(x)}\sqrt{\det \nabla f_\mu(y)}K_z\lp f_\mu(x)- f_\mu(y)\rp.
\end{equation}
The main task is therefore to extend the kernel \eqref{FreeDistortedKernel} to $\mu\in D(0,1/L)$ and $z\in\C\setminus\sigma_\mu$ (see \eqref{DistortedFreeSpectrum} for the definition of $\sigma_\mu$). The problem reduces to the terms $(z^2-1)^{1/2}$, $\sqrt{\det\nabla f_\mu(x)}$ and $\lv f_\mu(x)-f_\mu(y)\rv$.

Let $\mu\in D(0,1/L)$. We introduce the branch of the square root of $z^2-1$ naturally defined for $z\not\in\sigma_\mu$
\begin{equation}\label{RotatedBranch}
    \omega_\mu(z):=\frac{\lp(1+\mu)^2(z^2-1)\rp^{1/2}}{1+\mu}.
\end{equation}
We also set for $x\in\R^3$
\begin{equation}\label{JacobianContinuation}
    \rho_\mu(x):=\exp\lp\frac{1}{2}\tr\sum_{n=1}^{+\infty}\frac{(-1)^{n-1}(\mu \nabla F(x))^n}{n}\rp
\end{equation}
and for $y\not=x$
\begin{equation}\label{DistanceContinuation}
    \delta_\mu(x,y):=\sqrt{(f_\mu(x)- f_\mu(y))^\Tsf(f_\mu(x)- f_\mu(y))}
\end{equation}
where $\sqrt{\cdot}:\C\setminus(-\infty,0]\rightarrow\lcb\Re>0\rcb$. For real $\mu$ it is clear that $\omega_\mu(z)=(z^2-1)^{1/2}$, $\rho_\mu(x)=\sqrt{\det \nabla f_\mu(x)}$ and $\delta_\mu(x,y)=\lv f_\mu(x)- f_\mu(y)\rv$.
\begin{remark}\label{SquareRootDefinition}
    The square root in \eqref{DistanceContinuation} is well defined when $\mu$ is not real. Indeed if
    \[(f_\mu(x)-f_\mu(y))^\Tsf(f_\mu(x)- f_\mu(y))=\lv x-y\rv^2+2\mu(x-y)^\Tsf(F(x)-F(y))+\mu^2\lv F(x)-F(y)\rv^2\]
    is non-positive then, taking real and imaginary part, we would get
    \begin{equation*}
        \left\{
        \begin{array}{ll}
            \lv x-y\rv^2+2\Re\mu(x-y)^\Tsf(F(x)-F(y))+\lp(\Re\mu)^2-(\Im\mu)^2\rp\lv F(x)-F(y)\rv^2\leq 0\\
            
            (x-y)^\Tsf(F(x)-F(y))+\Re\mu\lv F(x)-F(y)\rv^2=0
        \end{array}
        \right.
    \end{equation*}
    which yields by combination
    \[\lv x-y\rv^2-\lv\mu\rv^2\lv F(x)-F(y)\rv^2\leq0\]
    and the contradiction since $F$ is $L-$Lipschitz and $\lv\mu\rv<1/L$. 
\end{remark}

The quantities $\omega_\mu$, $\rho_\mu$ and $\delta_\mu$ naturally lead to a candidate for the resolvent of $\Dvec_\mu$ outside $\sigma_\mu$. The following lemma shows that this candidate is indeed well defined and has the analyticity properties needed to justify this identification. Its proof is postponed to the appendix.

\begin{lemma}\label{TAnalyticity}
    For all $\mu\in D(0,1/L)$ and $z\not\in\sigma_\mu$ the integral operator $T_\mu^z$ with kernel 
    \[K_\mu^z:(x,y)\mapsto \rho_\mu(x)\rho_\mu(y)\frac{e^{ i\omega_\mu(z) \delta_\mu(x,y)}}{4\pi\delta_\mu(x,y)}\lp zI_4+\beta+\lp i+\omega_\mu(z) \delta_\mu(x,y)\rp\alpha\cdot\frac{ f_\mu(x)- f_\mu(y)}{\delta_\mu(x,y)^2}\rp\] 
    is well defined and bounded from $ L^2$ into itself. Moreover
    \begin{enumerate}
        \item for all $z\in(-1,1)$, $\mu\in D(0,1/L)\mapsto T_\mu^z\in\Lcal(L^2)$ is analytic,
        \item for all $\mu\in D(0,1/L)$, $z\in\C\setminus\sigma_\mu\mapsto T_\mu^z\in\Lcal(L^2)$ is analytic.
    \end{enumerate}
\end{lemma} 

We can now state and prove the main result of this section.

\begin{proposition}\label{DistortedResolvent}
    For all $\mu\in D(0,1/L)$, if $z\not\in\sigma_\mu$ then $z\not\in\Sp(\Dvec_\mu)$ and $(\Dvec_\mu-z)^{-1}=T_\mu^z$.
\end{proposition}

\begin{proof}
    For $z\in(-1,1)$ and $\mu\in(-1/L,1/L)$, one has on $H^1$ 
    \begin{equation}\label{LeftInverse}
        T_\mu^z(\Dvec_\mu-z)=1.
    \end{equation}
    Since $\mu\mapsto \Dvec_\mu-z$ is analytic from $D(0,1/L)$ to $\Lcal(H^1,L^2)$ and thanks to the first point of Lemma \ref{TAnalyticity}, \eqref{LeftInverse} remains true for all $\mu\in D(0,1/L)$ by analytic continuation. Thanks to the first point of Lemma \ref{TAnalyticity}, \eqref{LeftInverse} remains true for all $z\in\C\setminus\sigma_\mu$ by analytic continuation. 
    
    Let $\phi\in L^2$. For $z\in(-1,1)$ and $\mu\in(-1/L,1/L)$, $T_\mu^z\phi\in H^1$ and
    \begin{equation*}\label{RightInverse}
        (\Dvec_\mu-z)T_\mu^z\phi=\phi
    \end{equation*}
    Then for all $\psi\in H^1$ 
    \begin{equation}\label{NoName}
        \la(\Dvec_{\overline{\mu}}-\overline{z})\psi,T_\mu^z\phi\ra_{L^2}=\la\psi,\phi\ra_{L^2}.
    \end{equation}
    By analytic continuation, \eqref{NoName} remains true for all $\mu\in D(0,1/L)$ and $z\in\C\setminus\sigma_\mu$. Thanks to Proposition \ref{FreeDistortedAdjoint} and \eqref{NoName} we conclude that $T_\mu^z\phi\in H^1$ and $(\Dvec_\mu-z)T_\mu^z\phi=\phi$. 
\end{proof}

The distortion acting only outside $\lcb F=0\rcb$, one expects the distorted free resolvent to remain unchanged when localized in that region. The following Corollary makes this observation precise and provides a key tool for the analysis of the perturbed operator. Denote by $C_\mu$ the connected component of the complement of $\sigma_\mu\cup\sigma_0$ containing $0$ (see Figure \ref{Cmu}).
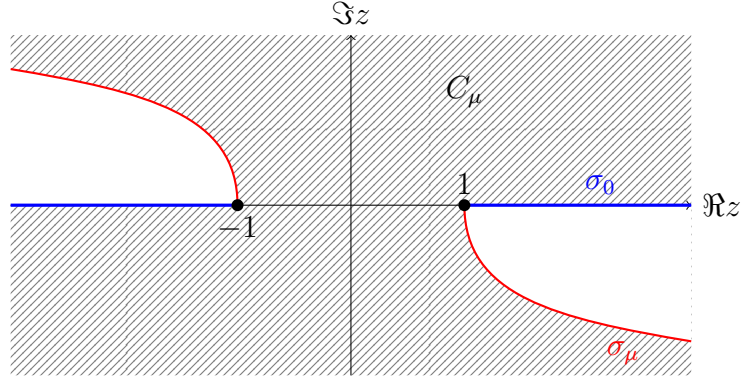
\begin{figure}[htbp]
    \centering
    \begin{tikzpicture}[scale=1.5]
        \fill[
            pattern=north east lines,
            pattern color=gray,
            even odd rule
        ]
        (-3,-1.5) rectangle (3,1.5)
        
        (-3,0)
        -- (-1,0)
        .. controls (-1,0.8) and (-1.8,1.0) .. (-3,1.2)
        -- cycle
        
        (1,0)
        -- (3,0)
        -- (3,-1.2)
        .. controls (1.8,-1.0) and (1,-0.8) .. (1,0)
        -- cycle;
        
        \draw[->] (-3,0) -- (3,0) node[right] {$\Re z$};
        \draw[->] (0,-1.5) -- (0,1.5) node[above] {$\Im z$};
        
        \draw[very thick, blue] (-3,0) -- (-1,0);
        \draw[very thick, blue] (1,0) -- (3,0);
        
        \node[above, blue] at (2.2,0) {$\sigma_0$};
        
        \draw[thick,red,-]
          (1,0)
          .. controls (1,-0.8) and (1.8,-1.0)
          .. (3,-1.2);
        
        \draw[thick,red,-]
          (-1,0)
          .. controls (-1,0.8) and (-1.8,1.0)
          .. (-3,1.2);
        
        \node[red] at (2.4,-1.3) {$\sigma_\mu$};
        
        \fill (-1,0) circle (1.5pt);
        \fill (1,0) circle (1.5pt);
        
        \node[below] at (-1,0) {$-1$};
        \node[above] at (1,0) {$1$};
        
        \node at (1,1.0) {$C_\mu$};
    \end{tikzpicture}
    \caption{The domain $C_\mu$ when $\Im\mu>0$.}
    \label{Cmu}
\end{figure}

\begin{corollary}
    For all $\mu\in D(0,1/L)$, $z\in C_\mu$ and $\chi\in C^1_c$ supported in $\lcb F=0\rcb$ 
    \begin{equation}\label{CutoffFreeDistorted}
        \chi(\Dvec_\mu-z)^{-1}\chi=\chi(\Dvec-z)^{-1}\chi
    \end{equation}
\end{corollary}

\begin{proof}
    One has \eqref{CutoffFreeDistorted} for $z\in(-1,1)$ and $\mu\in(-1/L,1/L)$ since $\chi U_\mu=\chi=(\overline{\chi}U_\mu)^*$ and
    \[\chi(\Dvec_\mu-z)^{-1}\chi=\chi U_\mu(\Dvec-z)^{-1}(\overline{\chi}U_\mu)^*.\]
    By analytic continuation \eqref{CutoffFreeDistorted} remains true for all $\mu\in D(0,1/L)$ and then all $z\in C_\mu$.
\end{proof}

\section{Distorted Dirac operators with compactly supported perturbations}\label{Section3}

In this section, we are interested in the Dirac operator 
\begin{equation}\label{PerturbedDirac}
    H:=\Dvec+V
\end{equation}
where $V\in L^\infty_\comp(\R^3,\C^{4\times4})$ is not necessarily self-adjoint. It is clear that $H$ is closed in $L^2$ on its domain $H^1$ and that its essential spectrum is $\sigma_0$. Let us recall the definition of the resonances of $H$ and their multiplicity. Consider the two-sheeted Riemann surface
\begin{equation*}
    \Mpaz:=\lcb(z,\omega)\in\C^2\mbox{, } \omega^2=z^2-1\rcb.
\end{equation*}
There exists a finite meromorphic function from $\Mpaz$ to $\Lcal(L^2_\comp,H^1_\loc)$ denoted $\Rpaz$ such that if $(z,\omega)\in\Mpaz$ is not a pole of $\Rpaz$ and $\Im\omega>0$ then $z\not\in\Sp(H)$ and on $L^2_\comp$ 
\begin{equation}\label{ResolventContinuation}
    \Rpaz(z,\omega)=(H-z)^{-1}.
\end{equation} 
A resonance of $H$ is a pole $(\lambda,\kappa)$ of $\Rpaz$. When $\kappa\not=0$ its multiplicity is the rank of the residue of $\Rpaz$ in any local chart defined near $(\lambda,\kappa)$. For further details, see \cite{Du25}. We denote by $\Res(H)$ the resonance set of $H$.

We define the distorted Dirac operator by replacing $\Dvec$ by $\Dvec_\mu$ in \eqref{PerturbedDirac}
\begin{equation*}
    H_\mu:=\Dvec_\mu+V.
\end{equation*}
and assume in all the section that $F=0$ on a neighborhood of $\Supp V$ (see Figure \ref{FGraph}). 

\subsection{Resonance-eigenvalue correspondence} In this subsection we prove the  
\begin{theorem}\label{LinkResonanceEigenvalue}
     Let $\mu\in D(0,1/L)$ and $(\lambda,\kappa)\in\Mpaz$ with $\Im(1+\mu)\kappa>0$. Then $(\lambda,\kappa)\in\Res(H)$ if and only if $\lambda\in\Sp_\disc(H_\mu)$, with the same multiplicity.
\end{theorem}

We wish to identify suitable localizations of $\Rpaz$ and resolvent of $H_\mu$. In fact, the distortion acting outside the support of the potential, one expects \eqref{CutoffFreeDistorted} to hold for the perturbed operators. Proposition 3.1 makes this precise.  The remaining issue is the preservation of multiplicities. We must show that localization near the support of the potential does not alter the ranks of the corresponding residues. Proposition 3.2 makes this precise.

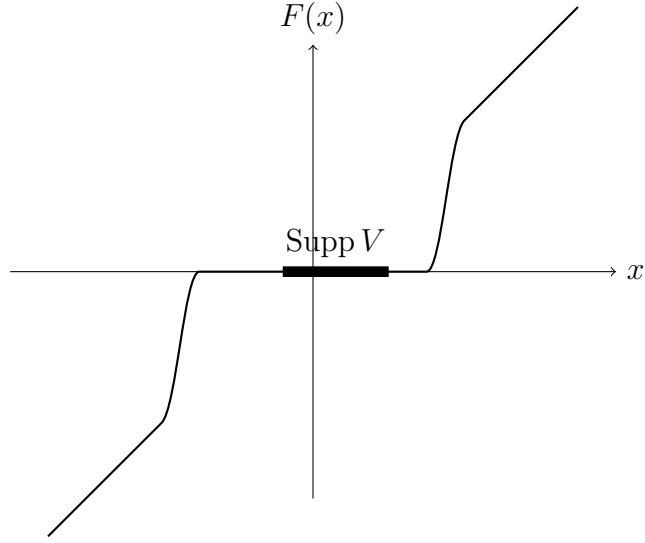
\begin{figure}[htbp]
    \begin{tikzpicture}[scale=1]
        \draw[->] (-4,0) -- (4,0) node[right] {$x$};
        \draw[->] (0,-3) -- (0,3) node[above] {$F(x)$};
        \draw[thick]
        (-3.5,-3.5) -- (-2,-2)
        .. controls (-1.8,-1.8) and (-1.7,0) .. (-1.5,0)
        -- (1.5,0)
        .. controls (1.7,0) and (1.8,1.8) .. (2,2)
        -- (3.5,3.5);
        \draw[line width=4pt] (-0.4,0) -- (1,0);
        \node[above] at (0.3,0.05) {$\Supp V$};
    \end{tikzpicture}
    \caption{Graph of the function $F$}
    \label{FGraph}
\end{figure}  

\begin{proposition}\label{DistortedSpectrum} 
    Let $\mu\in D(0,1/L)$.
    \begin{itemize}
        \item The spectrum of $H_\mu$ outside $\sigma_\mu$ is discrete.
        \item Given $z\in C_\mu$ one has $z\not\in\Sp(H_\mu)$ if and only if $z\not\in\Sp(H)$, in that case 
        \begin{equation}\label{CutoffDistorted}
            \chi(H_\mu-z)^{-1}\chi=\chi(H-z)^{-1}\chi   
        \end{equation}
        for $\chi\in C^1_c$ supported in $\lcb F=0\rcb$.
    \end{itemize}
\end{proposition}

\begin{proof} 
    Let $z\in\C\setminus\sigma_\mu$. From
    \[H_\mu-z=(1+V(\Dvec_\mu-z)^{-1})(\Dvec_\mu-z)\]
    we deduce that $H_\mu-z$ has a bounded inverse from $L^2$ to $H^1$ if and only if $1+V(\Dvec_\mu-z)^{-1}$ is invertible. The resolvent $(\Dvec_\mu-z)^{-1}$ being bounded from $L^2$ to $H^1$ and $V$ being compactly supported, Rellich-Kondrachov Theorem implies that the operator $V(\Dvec_\mu-z)^{-1}$ is compact from $L^2$ into itself. Let $\rho\in C^1_c$ supported in $\lcb F=0\rcb$ and such that $V\prec\rho$ i.e $\rho=1$ on a neighborhood of $\Supp V$. Consequently $\rho V=V$ so that
    \[1+V(\Dvec_\mu-z)^{-1}=(1+V(\Dvec_\mu-z)^{-1}(1-\rho))(1+V(\Dvec_\mu-z)^{-1}\rho)\]
    and $1+V(\Dvec_\mu-z)^{-1}(1-\rho)$ and $1-V(\Dvec_\mu-z)^{-1}(1-\rho)$ are inverse one of the other. Thanks to \eqref{CutoffFreeDistorted} if $z\in C_\mu$
    \begin{equation}\label{KeyPoint}
        V(\Dvec_\mu-z)^{-1}\rho=V(\Dvec-z)^{-1}\rho
    \end{equation}
    therefore if $\Im z\not=0$
    \[\lV V(\Dvec_\mu-z)^{-1}\rho\rV_{\Lcal(L^2)}\lesssim\frac{1}{\lv\Im z\rv}\]
    and hence $1+V(\Dvec_\mu-z)^{-1}\rho$ is invertible for $\lv\Im z\rv\gg1$ thanks to Neumann invertibility criterion. We conclude that $1+V(\Dvec_\mu-z)^{-1}$ is invertible for such $z$. By the analytic Fredholm Theorem $z\mapsto(1+V(\Dvec_\mu-z)^{-1})^{-1}$ is finite meromorphic from $\C\setminus\sigma_\mu$ to $\Lcal(L^2)$ with poles set
    \[\lcb\lambda\in\C\setminus\sigma_\mu\mbox{, }1+V(\Dvec_\mu-\lambda)^{-1} \mbox{ is not invertible}\rcb=\Sp(H_\mu)\cap(\C\setminus\sigma_\mu).\]
    If $\lambda\in\Sp(H_\mu)\cap(\C\setminus\sigma_\mu)$ then $\lambda$ is isolated in $\Sp(H_\mu)$ and for $0<\lv z-\lambda\rv\ll 1$ 
    \[(H_\mu-z)^{-1}=(\Dvec_\mu-z)^{-1}(1+V(\Dvec_\mu-z)^{-1})^{-1}\]
    which guarantees that $\lambda$ is a pole of finite type of $z\mapsto(H_\mu-z)^{-1}$. Hence every point of $\Sp(H_\mu)\cap(\C\setminus\sigma_\mu)$ is an isolated eigenvalue of finite algebraic multiplicity of $H_\mu$, proving the first statement. 

    Considering the argument with $\mu=0$ and \eqref{KeyPoint} we conclude that $H_\mu$ and $H$ have same spectrum in $C_\mu$. Finally let $z\in C_\mu$ with $z\not\in\Sp(H_\mu)$ and $\chi\in C^1_c$ supported in $\lcb F=0\rcb$. Since 
    \[(1+V(\Dvec_\mu-z)^{-1})^{-1}=(1+V(\Dvec_\mu-z)^{-1}\rho)^{-1}(1-V(\Dvec_\mu-z)^{-1}(1-\rho))\]
    then with $\Tilde{\chi}\in C^1_c$ supported in $\lcb F=0\rcb$ such that $\chi,\rho\prec\Tilde{\chi}$ and using again \eqref{CutoffFreeDistorted} we get 
    \begin{equation*}
        \left\{
        \begin{array}{ll}
            (1-V(\Dvec_\mu-z)^{-1}(1-\rho))\chi=\Tilde{\chi}(1-V(\Dvec-z)^{-1}(1-\rho))\chi\\
            
            \lp1+V(\Dvec_\mu-z)^{-1}\rho\rp^{-1}\Tilde{\chi}=\Tilde{\chi}\lp1+V(\Dvec-z)^{-1}\rho\rp^{-1}
        \end{array}
        \right.
    \end{equation*}
    and then
    \[\chi(H_\mu-z)^{-1}\chi=\chi(\Dvec-z)^{-1}\Tilde{\chi}\lp1+V(\Dvec-z)^{-1}\rho\rp^{-1}\lp1-V(\Dvec-z)^{-1}(1-\rho)\rp\chi.\]
    Again this holds in particular for $\mu=0$ and we eventually get \eqref{CutoffDistorted}.
\end{proof}

Since resonances of $H$ are points of the Riemann surface $\Mpaz$, whereas the eigenvalues of $H_\mu$ are complex numbers, we must identify an appropriate local chart of $\Mpaz$. We use the parametrization given by the branch $\omega_\mu$ (see \eqref{RotatedBranch} for its definition). A point $(\lambda,\kappa)\in\Mpaz$ satisfies $\Im(1+\mu)\kappa>0$ if and only if $\kappa=\omega_\mu(\lambda)$. In that case $(\lambda,\kappa)\in\Res(H)$ if and only if $\lambda$ is a pole of $z\mapsto\Rpaz(z,\omega_\mu(z))$ and its multiplicity is the rank of the corresponding residue. 

Before turning to multiplicity issues, we first record a simple relation between the distorted operators and their adjoints that will be used repeatedly below. Thanks to Proposition \ref{FreeDistortedAdjoint}, with domains $H^1$ one has 
\begin{equation}\label{PerturbedDistortedAdjoint}
    H_\mu^*=\Dvec_{\overline{\mu}}+V^*.
\end{equation}

\begin{proposition}\label{TruncatedRanks}
    Let $\chi\in C^1_c$ such that $V\prec\chi$. 
    \begin{enumerate}
        \item If $\lambda\in\Sp_\disc(H_\mu)$ and $\Pi$ is the associated Riesz projection then $\rank(\Pi)=\rank(\chi\Pi\chi)$.
        \item If $(\lambda,\kappa)\in\Res(H)$ with $\Im(1+\mu)\kappa>0$ and $\Ppaz$ is the residue of $z\mapsto\Rpaz(z,\omega_\mu(z))$ at $\lambda$ then $\rank(\Ppaz)=\rank(\chi\Ppaz\chi)$.
    \end{enumerate}
\end{proposition}

\begin{proof} 
    \textbf{Statement (1).} We first prove that the multiplication by $\chi$ is one-to-one on $\Ran(\Pi)$. Let us recall that the endomorphism
    \[N:\psi\in\Ran(\Pi)\mapsto(H_\mu-\lambda)\psi\in\Ran(\Pi)\]
    is well defined and nilpotent. Let $\psi\in\Ran(\Pi)$ such that $\chi\psi=0$. Let $m$ be the first integer such that $N^m\psi=0$. Assume that $m>0$ so that $N^{m-1}\psi\not=0$. Since $\psi=0$ in a neighborhood of $\Supp V$ then $N\psi$ does so. Iterating, we get $VN^{m-1}\psi=0$. Finally
    \begin{align*}
        (\Dvec_\mu-\lambda)N^{m-1}\psi&=(H_\mu-\lambda)N^{m-1}\psi\\
        &=0
    \end{align*}
    therefore $N^{m-1}\psi$ is an eigenvector $\Dvec_\mu$ which is a contradiction thanks to Proposition \ref{DistortedResolvent}. We conclude that $m=0$ i.e $\psi=0$. 
    
    Similarly the multiplication by $\overline{\chi}$ is one-to-one on $\Ran(\Pi^*)$. Indeed $\overline{\lambda}$ is an isolated eigenvalue of finite multiplicity of $H_\mu^*$ with associated Riesz-projection $\Pi^*$, \eqref{PerturbedDistortedAdjoint} holds and $V^*\prec\overline{\chi}$ therefore the same proof applies. We conclude that $\overline{\chi}\Pi^*$ and $\Pi^*$ have same rank.  Since a bounded operator and its adjoint have same rank we deduce $\Ran(\Pi\chi)=\Ran(\Pi)$.
    
    Combining these two properties, we conclude that the multiplication by $\chi$ is invertible from $\Ran(\Pi)$ to $\Ran(\chi\Pi\chi)$ therefore the conclusion is straightforward.
    
    \textbf{Statement (2).} We first show that the multiplication by $\chi$ is one-to-one on $\Ran(\Ppaz)$ following similar arguments. The endomorphism
    \[\Npaz :\psi\in\Ran(\Ppaz)\mapsto(\beta-i\alpha\cdot\nabla+V-\lambda)\psi\in\Ran(\Ppaz)\]
    is well defined and nilpotent. Let $\psi\in\Ran(\Ppaz)$ such that $\chi\psi=0$. Assume that the first integer $m$ such that $\Npaz^m\psi=0$ is positive. Then $\Npaz^{m-1}\psi=0$ in a neighborhood of $\Supp V$ and $(\beta-i\alpha\cdot\nabla-\lambda)\Npaz^{m-1}\psi=0$. This yields $(-\Delta-\kappa^2)\Npaz^{m-1}\psi=0$ and $\Npaz^{m-1}\psi=0$ by unique continuation \cite[Theorem XIII.63]{ReSi72}. Eventually $\psi=0$.
    
    We recall that if $I\geq 1$ is the index of $\Npaz$ then
    \[z\mapsto \Rpaz(z,\omega_\mu(z))-\sum_{i=1}^I\frac{\Npaz^{i-1}\Ppaz}{(z-\lambda)^i}\]
    has a holomorphic continuation from an open neighbourhood of $\lambda$ to $\Lcal(L^2_\comp,H^1_\loc)$. Moreover there exists a holomorphic function from $\Mpaz$ to $\Lcal(L^2_\comp,H^1_\loc)$ denoted $\Rpaz_0$ such that for $(z,\omega)\in\Mpaz$ with $\Im\omega>0$ one has $z\not\in\sigma_0$ and on $L^2_\comp$
    \begin{equation}\label{FreeResolventContinuation}
        \Rpaz_0(z,\omega)=(\Dvec-z)^{-1}.
    \end{equation}
    For all $(z,\omega)\in\Mpaz\setminus\Res(H)$ one has
    \begin{equation}\label{ResolventEquation}
        \Rpaz(z,\omega)=\Rpaz_0(z,\omega)-\Rpaz(z,\omega)V\Rpaz_0(z,\omega).
    \end{equation}
    Indeed \eqref{ResolventEquation} is standard when $\Im\omega>0$ thanks to \eqref{ResolventContinuation} and \eqref{FreeResolventContinuation} and then extended by analytic continuation. Setting $\omega=\omega_\mu(z)$ in \eqref{ResolventEquation} and comparing the residues at $\lambda$ we get
    \begin{equation}\label{ResiduesComparison}
        \Ppaz=-\sum_{i=0}^{I-1}\Npaz ^i\Ppaz V\dfrac{d^i}{dz^i}_{\tq z=\lambda} \Rpaz_0(z,\omega_\mu(z)).
    \end{equation}
    Let $\chi_1\in C^1_c$ such that $V\prec\chi_1\prec\chi$. On $ H^1$ one has 
    \[\Npaz \Ppaz\chi_1=\Ppaz(H-\lambda)\chi_1=\Ppaz\chi(H-\lambda)\chi_1\]
    therefore by density
    \[\Ran(\Npaz \Ppaz\chi_1)\subset\Ran(\Ppaz\chi).\]
    Iterating, we construct $\chi_2,...,\chi_{I-1}\in C^1_c$ such that $V\prec\chi_{I-1}\prec...\prec\chi_1\prec\chi$ and
    \[\Ran(\Npaz ^{I-1}\Ppaz\chi_{I-1})\subset...\subset\Ran(\Npaz\Ppaz\chi_1)\subset\Ran(\Ppaz\chi).\]
    In particular $\Ran(\Npaz^i\Ppaz V)\subset\Ran(\Ppaz\chi)$ therefore \eqref{ResiduesComparison} yields $\Ran(\Ppaz)=\Ran(\Ppaz\chi)$. We then conclude as for the first statement. 
\end{proof}

Finally, we are ready for the proof of Theorem \ref{LinkResonanceEigenvalue}. 

\begin{proof}
    Let $\chi\in C^1_c$ supported in $\lcb F=0\rcb$ such that $V\prec\chi$. Let us show that for all $z\in\C\setminus\sigma_\mu$ such that $(z,\omega_\mu(z))\not\in\Res(H)$ and $z\not\in\Sp_\disc(H_\mu)$
    \begin{equation}\label{LinkDistortedOutgoingResolvent}
        \chi\Rpaz(z,\omega_\mu(z))\chi=\chi(H_\mu-z)^{-1}\chi.
    \end{equation}
    If $z\in(-1,1)$ then $\omega_\mu(z)$ is a square root of $z^2-1\leq 0$ therefore $\Re\omega_\mu(z)=0$ and 
    \begin{align*}
        0&<\Im(1+\mu)\omega_\mu(z)\\
        &=(1+\Re\mu)\Im\omega_\mu(z)
    \end{align*}
    hence $\Im\omega_\mu(z)>0$ since $1+\Re\mu>0$. Thanks to \eqref{ResolventContinuation} and \eqref{CutoffDistorted} we conclude that \eqref{LinkDistortedOutgoingResolvent} holds first when $z\in(-1,1)$ and then $z\in\C\setminus\sigma_\mu$ by analytic continuation.
    
    The combination of \eqref{LinkDistortedOutgoingResolvent} and Proposition \ref{TruncatedRanks} ends the proof of Theorem \ref{LinkResonanceEigenvalue}.
\end{proof} 

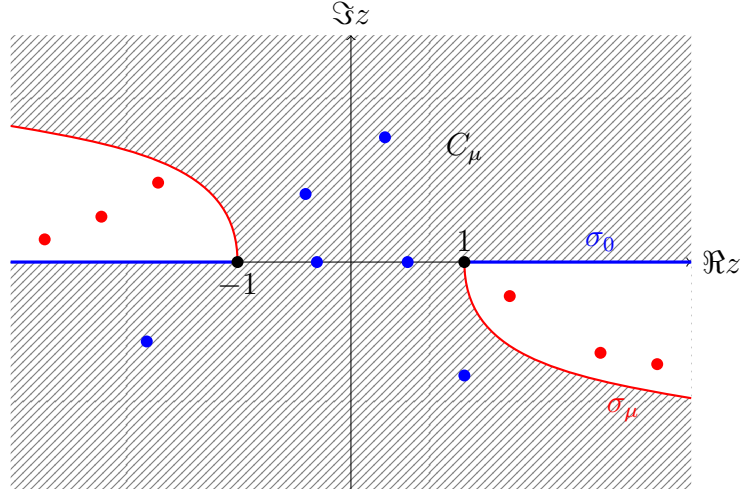
\begin{figure}[htbp]
    \centering
    \begin{tikzpicture}[scale=1.5]
        \fill[
            pattern=north east lines,
            pattern color=gray,
            even odd rule
        ]
        (-3,-2) rectangle (3,2)
        
        (-3,0)
        -- (-1,0)
        .. controls (-1,0.8) and (-1.8,1.0) .. (-3,1.2)
        -- cycle
        
        (1,0)
        -- (3,0)
        -- (3,-1.2)
        .. controls (1.8,-1.0) and (1,-0.8) .. (1,0)
        -- cycle;
        
        \draw[->] (-3,0) -- (3,0) node[right] {$\Re z$};
        \draw[->] (0,-2) -- (0,2) node[above] {$\Im z$};
        
        \draw[very thick, blue] (-3,0) -- (-1,0);
        \draw[very thick, blue] (1,0) -- (3,0);
        
        \node[above, blue] at (2.2,0) {$\sigma_0$};
        
        \draw[thick,red,-]
          (1,0)
          .. controls (1,-0.8) and (1.8,-1.0)
          .. (3,-1.2);
        
        \draw[thick,red,-]
          (-1,0)
          .. controls (-1,0.8) and (-1.8,1.0)
          .. (-3,1.2);
        
        \node[red] at (2.4,-1.3) {$\sigma_\mu$};
        
        \fill (-1,0) circle (1.5pt);
        \fill (1,0) circle (1.5pt);
        
        \node[below] at (-1,0) {$-1$};
        \node[above] at (1,0) {$1$};
        
        \node at (1,1.0) {$C_\mu$};

        \fill[blue] (-0.4,0.6) circle (1.5pt);
        \fill[blue] (0.3,1.1) circle (1.5pt);
        \fill[blue] (-1.8,-0.7) circle (1.5pt);
        \fill[blue] (1,-1) circle (1.5pt);
        \fill[blue] (0.5,0) circle (1.5pt);
        \fill[blue] (-0.3,0) circle (1.5pt);
        
        \fill[red] (-1.7,0.7) circle (1.5pt);
        \fill[red] (-2.2,0.4) circle (1.5pt);
        \fill[red] (-2.7,0.2) circle (1.5pt);

        \fill[red] (1.4,-0.3) circle (1.5pt);
        \fill[red] (2.2,-0.8) circle (1.5pt);
        \fill[red] (2.7,-0.9) circle (1.5pt);
        
    \end{tikzpicture}
    \caption{The colored points describe $\lambda\in\Sp_\disc(H_\mu)$. In blue, $\lambda$ is an eigenvalue of $H$. In red, $(\lambda,\kappa)$ is a resonance of $H$ with $\Im(1+\mu)\kappa>0$.}
    \label{SpectralPicture}
\end{figure}

\subsection{Application to resonance multiplicity.} Until the end of the section, we wish to apply Theorem \ref{LinkResonanceEigenvalue} to establish the even multiplicity of resonances. Since resonance multiplicities are intrinsically defined, the conclusion should not depend on the distortion used to realize resonances as eigenvalues. The next two Lemmas successively remove the dependence on the distortion parameter and on the deformation.

Fix $F$ and recall that $L=\lV\nabla F\rV_{L^\infty}\geq 1$. Theorem \ref{LinkResonanceEigenvalue} applies to resonances $(\lambda,\kappa)\in\Mpaz$ satisfying $\Im(1+\mu)\kappa>0$ for a given $\mu\in D(0,1/L)$. The next Lemma characterizes the resonances for which such a $\mu$ exists. 

\begin{lemma}\label{CatchedResonances}
    Given $\ell\geq1$ and $\kappa\in\C$, there exists $\mu\in D(0,1/\ell)$ such that $\Im(1+\mu)\kappa>0$ if and only if $\Im\kappa>0$ or $\lv\Re\kappa\rv>\sqrt{\ell^2-1}\lv\Im\kappa\rv$.
\end{lemma}

\begin{proof}
    Assume that $\Im(1+\mu)\kappa>0$ for some $\mu\in D(0,1/ \ell)$. In particular $\lv\kappa\rv>- \ell\Im\kappa$ therefore if $\Im\kappa\leq 0$ then $(\Im\kappa)^2+(\Re\kappa)^2\geq  \ell^2(\Im\kappa)^2$. This provides the direct implication. 
    
    Conversely, we argue by contraposition. Assume that $\Im(1+\mu)\kappa\leq0$ for all $\mu\in D(0,1/ \ell)$. We can assume $\kappa\not=0$ and choose $\mu=i(1/ \ell-\varepsilon)\overline{\kappa}/\lv\kappa\rv$ for some $\varepsilon\in(0,1/ \ell)$ in such a way that $\Im\kappa+(1/ \ell-\varepsilon)\lv\kappa\rv\leq0$. Letting $\varepsilon\to0$ we conclude $ \ell\Im\kappa\leq-\lv\kappa\rv$. In particular $\Im\kappa\leq 0$ and $(\Im\kappa)^2+(\Re\kappa)^2\leq  \ell^2(\Im\kappa)^2$. This ends the proof of the Lemma.
\end{proof}

The following lemma shows that the deformation can be chosen with Lipschitz constant arbitrarily close to the optimal value one.

\begin{lemma}\label{FChoice}
    For all compact set $K\subset\R^3$ and $\ell>1$ there exists $F\in C^2(\R^3,\R^3)$ satisfying
    \begin{equation*}
        \left\{
        \begin{array}{ll}
            F=0 \mbox{ on a neighborhood of } K\\

            F(x)=x \mbox{ if } \lv x\rv\gg 1\\
            
            \lV \nabla F\rV_{L^\infty}\leq\ell
        \end{array}
        \right.
    \end{equation*}
\end{lemma}

\begin{proof}
    Let $r_0>0$ such that $K\subset B(0,r_0)$. We set $F(x)=f(\lv x\rv)x$ where $f:\R\rightarrow[0,1]$ is $C^2$ and
    \begin{equation*}
        f(r):=\left\{
        \begin{array}{ll}
            0 &\mbox{ if } r\in (-\infty,r_0)\\
            
            1 &\mbox{ if } r\in(r_1,+\infty)
        \end{array}.
        \right.
    \end{equation*}
    for some $r_1>0$ chosen later. We compute for $x\not=0$
    \begin{equation*}
        \nabla F(x)=f(\lv x\rv)I_3+\frac{f'(\lv x\rv)}{\lv x\rv}xx^\Tsf
    \end{equation*}
    therefore 
    \begin{equation*}
        \lV\nabla F\rV_{L^\infty}\leq 1+\sup_{r\geq0}\lv rf'(r)\rv.
    \end{equation*}
    The problem reduces to the question : can we define $f$ on $[r_0,r_1]$ in such a way that 
    \begin{equation*}
        \sup_{r\geq0}\lv rf'(r)\rv\leq\ell-1\mbox{ ?}
    \end{equation*}
    This constraint indicates a logarithmic behaviour. To preserve smoothness, we introduce a $C^2-$smooth function $g:\R\rightarrow[0,1]$ such that
    \begin{equation*}
         g(r):=\left\{
        \begin{array}{ll}
            0 &\mbox{ if } r\in (-\infty,1/4)\\
            
            1 &\mbox{ if } r\in(3/4,+\infty)
        \end{array},
        \right.
    \end{equation*}
    and for $r\in[r_0,r_1]$ we set
    \begin{equation*}
        f(r):= g\lp\frac{\log r-\log r_0}{\log r_1-\log r_0}\rp.
    \end{equation*}
    That way, we get 
    \begin{equation*}
        f'(r)=\left\{
        \begin{array}{ll}
            0 &\mbox{ if } r\in (-\infty,r_0)\cup(r_1,+\infty)\\
            
            \dfrac{1}{r(\log r_1-\log r_0)} g'\lp\dfrac{\log r-\log r_0}{\log r_1-\log r_0}\rp &\mbox{ if } r\in[r_0,r_1]
        \end{array}
        \right.
    \end{equation*}
    therefore 
    \begin{equation}\label{LogarithmicUpper}
        \sup_{r\geq0}\lv rf'(r)\rv\leq\frac{\lV g'\rV_{L^\infty}}{\log r_1-\log r_0}.
    \end{equation}
    It remains to take $r_1$ large enough so that the r.h.s of \eqref{LogarithmicUpper} is smaller than $\ell-1>0$.
\end{proof}

We introduce the time-reversal operator, for $\psi\in L^2$
\begin{equation*}
    T\psi:=-i\gamma_5\alpha_2\overline{\psi}.
\end{equation*} 
where $\gamma_5:=-i\alpha_1\alpha_2\alpha_3$. We recall that $T$ is anti-linear, anti-unitary and satisfies $T^2=-1$. 

\begin{corollary}\label{KramersDegeneracy}
    Assume that $TV=V^*T$. Every resonance $(\lambda,\kappa)\in\Mpaz$ of $H$ such that $\kappa\not\in i(-\infty,0]$ has even multiplicity. 
\end{corollary}

\begin{remark}
    \begin{itemize}
        \item A resonance $(\lambda,\kappa)$ with $\kappa\in i(-\infty,0]$ is often called anti-bound state.
        
        \item The symmetry assumption includes electric potentials $V=vI_4$ ($v:\R^3\rightarrow\R$) or Lorentz scalar potentials $V=m\beta$ ($m:\R^3\to\R$) since $T\beta=\beta T$. On the other hand, magnetic perturbations $V=-\alpha\cdot A$ ($A:\R^3\rightarrow\R^3$) break time-reversal symmetry since $T\alpha_j=-\alpha_j T$.

        \item Corollary \ref{KramersDegeneracy} yields the even multiplicity of the eigenvalues of the (possibly non-self-adjoint) Dirac operator $H$.
    \end{itemize} 
\end{remark}

\begin{proof}
    If $\Im\kappa>0$ then $\Im(1+\mu)\kappa>0$ with $\mu=0$. If $\Im\kappa\leq 0$ then $\lv\Re\kappa\rv>\sqrt{\ell^2-1}\lv\Im\kappa\rv$ for some $\ell>1$ since $\kappa\not\in i(-\infty,0]$. We choose $F$ associated to $K=\Supp V$ and $\ell$ as in Lemma \ref{FChoice}. Thanks to Lemma \ref{CatchedResonances} there exists $\mu\in D(0,1/L)$ such that $\Im(1+\mu)\kappa>0$. Theorem \ref{LinkResonanceEigenvalue} ensures that $\lambda\in\Sp_\disc(H_\mu)$ and the multiplicity of the resonance $(\lambda,\kappa)$ is the rank of the associated Riesz-projection $\Pi$. Using $TV=V^*T$, $T\Dvec_\mu=\Dvec_{\overline{\mu}}T$ and \eqref{PerturbedDistortedAdjoint} we get 
    $(H_\mu^*-\overline{z})^{-1}T=T(H_\mu-z)^{-1}$ if $0<\lv z-\lambda\rv\ll1$ and hence
    \begin{equation}\label{TimeReversalResonant}
        \Pi^*T=T\Pi
    \end{equation}
    Pick a basis $(\psi_1,...,\psi_r)$ of $\Ran(\Pi)$ and set $M:=(\la \psi_i,T \psi_j\ra)_{ij}\in\C^{r\times r}$. Thanks to the properties of $T$ it is clear that $M^\Tsf=-M$. In particular $\det(M)=(-1)^r\det(M)$ therefore if $M$ is invertible then $r$ is even. Pick $\nu\in\C^r$ such that $M\nu=0$. With $\psi:=\sum_{j=1}^r\overline{\nu_j}\psi_j$, this exactly means that $T\psi$ is orthogonal to $\Ran(\Pi)$ and consequently $\Pi^*T\psi=0$. Thanks to \eqref{TimeReversalResonant} we deduce that 
    $T\Pi\psi=0$. Since $T$ is one-to-one and $\psi\in\Ran(\Pi)$ we conclude that $\psi=0$. Eventually $\nu=0$ and the expected invertibility is proven.
\end{proof}

\section{Distorted exterior MIT bag model}\label{Section4}

In this section, we extend the previous strategy to boundary value problems. In the Dirac setting, the canonical example is the exterior MIT bag model. Let $\Omega\subset\R^3$ be open, $ C^2$-smooth and such that $\R^3\setminus\Omega$ is bounded. The outward pointing normal to the boundary of $\Omega$ is denoted $n$ (see Figure \ref{FigureObstacle}).

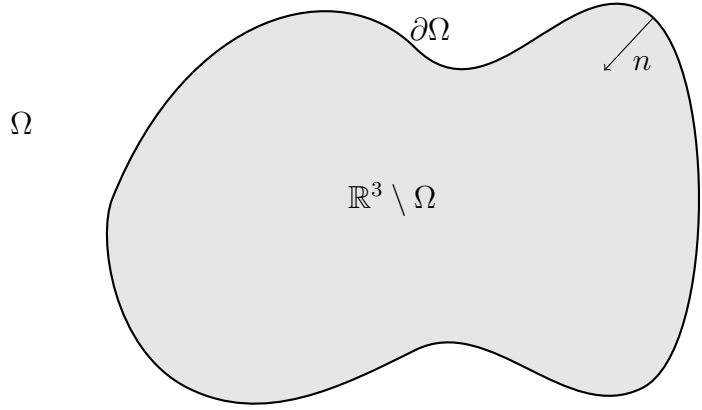
\begin{figure}[htbp]
    \centering
    \begin{tikzpicture}
        \fill[gray!20] 
          (0,0) 
          .. controls (1,2.5) and (3,3) .. (4,2)
          .. controls (5,1) and (6,3) .. (7,2.5)
          .. controls (8,2) and (8,-2) .. (7,-2.5)
          .. controls (6,-3) and (5,-1.5) .. (4,-2)
          .. controls (3,-2.5) and (2,-3) .. (1,-2.5)
          .. controls (0,-2) and (-0.2,-0.5) .. (0,0);
        
        \draw[thick] 
          (0,0) 
          .. controls (1,2.5) and (3,3) .. (4,2)
          .. controls (5,1) and (6,3) .. (7,2.5)
          .. controls (8,2) and (8,-2) .. (7,-2.5)
          .. controls (6,-3) and (5,-1.5) .. (4,-2)
          .. controls (3,-2.5) and (2,-3) .. (1,-2.5)
          .. controls (0,-2) and (-0.2,-0.5) .. (0,0);

        \draw[->] (7.15,2.4) -- (6.5,1.7);
        
        \node at (-1.2,1) {$\Omega$};
        \node at (3.7,0) {$\R^3\setminus\Omega$};
        \node at (4.2,2.25) {$\partial\Omega$};
        \node at (7.0,1.8) {$n$};
    \end{tikzpicture}
    \caption{The exterior domain.}
    \label{FigureObstacle}
\end{figure}

We are interested in the Dirac operator 
\begin{equation}\label{MIT}
    \Dvec^\Omega:=\beta-i\alpha\cdot\nabla \mbox{ with domain } \Dfrak:=\lcb\psi\in H^1(\Omega)\tq (-\imath\beta\alpha\cdot n)\psi_{\tq\partial\Omega}=\psi_{\tq\partial\Omega}\rcb.
\end{equation}
Notice that $\psi_{\tq\partial\Omega}\in L^2(\partial\Omega)$ denotes the trace of $\psi\in H^1(\Omega)$. We recall that $\Dvec^\Omega$ is self-adjoint in $L^2(\Omega)$ (see \cite[Theorem 3.2]{OuVe18}) and that its spectrum is purely essential and equal to $\sigma_0$ (see \cite[Theorem 3.1]{BeBrZr24}). Also $\Dfrak$ is a Banach space when equipped with $\lV\cdot\rV_{H^1(\Omega)}$. Let us recall the definiton of the resonances of $\Dvec^\Omega$ and their multiplicity (again see \cite{Du25}). There exists a finite meromorphic function from $\Mpaz$ to $\Lcal(L^2_\comp(\overline{\Omega}),\Dfrak_\loc)$ denoted $\Rpaz$ such that if $(z,\omega)\in\Mpaz$ and $\Im\omega>0$ then $(z,\omega)$ is not a pole of $\Rpaz$, $z\not\in\Sp(\Dvec^\Omega)$ and on $L^2_\comp(\overline{\Omega})$
\begin{equation}\label{MITResolventContinuation}
    \Rpaz(z,\omega)=(\Dvec^\Omega-z)^{-1}    
\end{equation}
A resonance of $\Dvec^\Omega$ is a pole $(\lambda,\kappa)$ of $\Rpaz$. We always have $\kappa\not=0$ and its multiplicity is the rank of the residue of $\Rpaz$ in any local chart defined near $(\lambda,\kappa)$. We denote by $\Res(\Dvec^\Omega)$ the resonance set of $\Dvec^\Omega$.

We define the distorted exterior MIT bag model by replacing $\nabla$ by $\nabla_\mu$ in \eqref{MIT} 
\begin{equation*}\label{DistortedMIT}
    \Dvec_\mu^\Omega:=\beta-i\alpha\cdot\nabla_\mu \mbox{ with domain } \Dfrak.
\end{equation*}
We assume in all this section that $F=0$ on a neighborhood of $\R^3\setminus\Omega$.

The goal of this section is to prove the
\begin{theorem}\label{MITLinkResonanceEigenvalue}
     Let $\mu\in D(0,1/L)$ and $(\lambda,\kappa)\in\Mpaz$ with $\Im(1+\mu)\kappa>0$. Then $(\lambda,\kappa)\in\Res(\Dvec^\Omega)$ if and only if $\lambda\in\Sp_\disc(\Dvec^\Omega_\mu)$, with the same multiplicity.
\end{theorem}

The proof of Theorem \ref{MITLinkResonanceEigenvalue} follows the same strategy as in the previous section. First we identify the localized resolvent.
    
\begin{proposition}\label{MITDistortedSpectrum} 
    Let $\mu\in D(0,1/L)$.
    \begin{itemize}
        \item The spectrum of $\Dvec_\mu^\Omega$ outside $\sigma_\mu$ is discrete.
        \item For $z\in C_\mu$ such that $z\not\in\Sp(\Dvec_\mu^\Omega)$ and $\chi\in C^1_c(\overline{\Omega})$ supported in $\lcb F=0\rcb$
        \begin{equation}\label{MITCutoffDistorted}
            \chi(\Dvec_\mu^\Omega-z)^{-1}\chi=\chi(\Dvec^\Omega-z)^{-1}\chi.
        \end{equation}
    \end{itemize}
\end{proposition}

In contrast with the case of compactly supported perturbations, the presence of the boundary will require the construction of both a right and a left parametrix, in order to obtain respectively surjectivity and injectivity modulo compact operators.

\begin{proof} 
    Let $z\not\in\sigma_\mu$. Pick $\rho_0,\rho_1\in C^2_c(\overline{\Omega})$ supported in $\lcb F=0\rcb$ such that $\indic_{\R^3\setminus\Omega}\prec\rho_0\prec\rho_1$, $z_0\not\in\sigma_0$ and set
    \begin{equation*}
        \left\{
        \begin{array}{ll}
            T_\mu(z):=(1-\rho_0)r_\Omega(\Dvec_\mu-z)^{-1}e_\Omega+\rho_1(\Dvec^\Omega-z_0)^{-1}\rho_0\\
            
            K_\mu(z):=i\alpha\cdot(\nabla\rho_0)r_\Omega(\Dvec_\mu-z)^{-1}e_\Omega+\lp i\alpha\cdot(\nabla\rho_1)-(z-z_0)\rho_1\rp(\Dvec^\Omega-z_0)^{-1}\rho_0
        \end{array}
        \right.
    \end{equation*}
    where $r_\Omega$ is the restriction operator to $\Omega$ and $e_\Omega$ is the extension operator by $0$ outside $\Omega$. Notice that $T_\mu(z)$ is bounded from $L^2(\Omega)$ into $\Dfrak$. We obtain a right parametrix on $L^2(\Omega)$
    \begin{equation}\label{RightParametrix}
        (\Dvec_\mu^\Omega-z)T_\mu(z)=1+K_\mu(z).
    \end{equation}

    Pick $\rho_2\in C^2_c(\overline{\Omega})$ supported in $\lcb F=0\rcb$ such that $\rho_1\prec\rho_2$ and set 
    \begin{equation*}
        \left\{
        \begin{array}{ll}
            \tilde{T}_\mu(z):=(1-\rho_1)r_\Omega(\Dvec_\mu-z)^{-1}e_\Omega(1-\rho_0)+\rho_1(\Dvec^\Omega-z_0)^{-1}\rho_2\\
            
            \tilde{K}_\mu(z):=-i(1-\rho_1)r_\Omega(\Dvec_\mu-z)^{-1}e_\Omega\alpha\cdot(\nabla\rho_0)+\rho_1(\Dvec^\Omega-z_0)^{-1}\lp i\alpha\cdot(\nabla\rho_2)-(z-z_0)\rho_2\rp
        \end{array}
        \right.
    \end{equation*}
    We obtain a left parametrix on $\Dfrak$
    \begin{equation}\label{LeftParametrix}
        \tilde{T}_\mu(z)(\Dvec_\mu^\Omega-z)=1+\tilde{K}_\mu(z).
    \end{equation}
    
    Since $\Dfrak$ is continuously embedded in $H^1(\Omega)$ and thanks to Rellich-Kondrachov Theorem the operators $K_\mu(z)$ and $\tilde{K}_\mu(z)$ are compact on $L^2(\Omega)$ and $\Dfrak$ respectively. One has
    \[1+K_\mu(z)=(1+K_\mu(z)(1-\rho_2))(1+K_\mu(z)\rho_2)\]
    and $1+K_\mu(z)(1-\rho_2)$ and $1-K_\mu(z)(1-\rho_2)$ are inverse one of the other. We take $z_0\in C_\mu$ in such a way that \eqref{CutoffFreeDistorted} yields
    \begin{equation*}
        K_\mu(z_0)\rho_2:=i\alpha\cdot(\nabla\rho_0)r_\Omega(\Dvec-z)^{-1}e_\Omega\rho_2+ i\alpha\cdot(\nabla\rho_1)(\Dvec^\Omega-z_0)^{-1}\rho_2
    \end{equation*}
    and then
    \begin{align*}
        \lV K_\mu(z_0)\rho_2\rV_{\Lcal(L^2(\Omega))}&\lesssim\lV(\Dvec-z_0)^{-1}\rV_{\Lcal(L^2)}+\lV(\Dvec^\Omega-z_0)^{-1}\rV_{\Lcal(L^2(\Omega))}\\
        &\lesssim\frac{1}{\lv\Im z_0\rv}
    \end{align*}
    which implies that $1+K_\mu(z_0)\rho_2$ is invertible on $L^2(\Omega)$ for $\lv\Im z_0\rv\gg1$ thanks to Neumann invertibility criterion. We conclude that $1+K_\mu(z_0)$ is invertible on $L^2(\Omega)$. In a very similar fashion, $1+\tilde{K}_\mu(z_0)$ is invertible on $L^2(\Omega)$. In addition, since $\tilde{K}_\mu(z_0)$ maps $L^2(\Omega)$ into $\Dfrak$ we conclude that $1+\tilde{K}_\mu(z_0)$ is invertible on $\Dfrak$.  Thanks to the analytic Fredholm Theorem, $z\mapsto(1+K_\mu(z))^{-1}$ is finite meromorphic from $\C\setminus\sigma_\mu$ to $\Lcal(L^2(\Omega))$ with poles set
    \[Z_\mu:=\lcb\lambda\in\C\setminus\sigma_\mu\mbox{, }1+K_\mu(\lambda) \mbox{ is not invertible on } L^2(\Omega)\rcb\]
    and $z\mapsto(1+\tilde{K}_\mu(z))^{-1}$ is finite meromorphic from $\C\setminus\sigma_\mu$ to $\Lcal(\Dfrak)$ with poles set
    \[\tilde{Z}_\mu:=\lcb\lambda\in\C\setminus\sigma_\mu\mbox{, }1+\tilde{K}_\mu(\lambda) \mbox{ is not invertible on } \Dfrak\rcb.\]

    If $\lambda\in\Sp(\Dvec^\Omega_\mu)\cap(\C\setminus\sigma_\mu)$ then $\lambda\in Z_\mu\cup\tilde{Z}_\mu$ thanks to \eqref{RightParametrix} and \eqref{LeftParametrix} therefore $\lambda$ is isolated in $\Sp(\Dvec^\Omega_\mu)$ and for $0<\lv z-\lambda\rv\ll 1$ 
    \[(\Dvec^\Omega_\mu-z)^{-1}=T_\mu(z)(1+K_\mu(z))^{-1}\]
    which guarantees that $\lambda$ is a pole of finite type of $z\mapsto(\Dvec^\Omega_\mu-z)^{-1}$. This ends the proof of the first statement. 

    Finally let $z\in C_\mu$ and $\chi\in C^1_c(\overline{\Omega})$ supported in $\lcb F=0\rcb$. For $z\not\in Z_\mu\cup\tilde{Z}_\mu$
    \[(1+K_\mu(z))^{-1}=(1+K_\mu(z)\rho_2)^{-1}(1-K_\mu(z)(1-\rho_2))\]
    then with $\Tilde{\chi}\in C^1_c(\overline{\Omega})$ supported in 
    $\lcb F=0\rcb$ such that $\chi,\rho_2\prec\Tilde{\chi}$ and using  \eqref{CutoffFreeDistorted} we get 
    \begin{equation*}
        \left\{
        \begin{array}{ll}
            (1-K_\mu(z)(1-\rho_2))\chi=\Tilde{\chi}(1-K_0(z)(1-\rho_2))\chi\\
            
            \lp1+K_\mu(z)\rho_2\rp^{-1}\Tilde{\chi}=\Tilde{\chi}\lp1+K_0(z)\rho_2\rp^{-1}\\

            \chi T_\mu(z)\tilde{\chi}=\chi T_0(z)\tilde{\chi}
        \end{array}
        \right.
    \end{equation*}
    and then
    \[\chi(\Dvec^\Omega_\mu-z)^{-1}\chi=\chi T_0(z)\Tilde{\chi}\lp1+K_0(z)\rho_2\rp^{-1}\lp1-K_0(z)(1-\rho_2)\rp\chi.\]
    This holds in particular for $\mu=0$ therefore we get \eqref{MITCutoffDistorted} for all $z\in C_\mu$ such that $z\not\in Z_\mu\cup\tilde{Z}_\mu$. Finally \eqref{MITCutoffDistorted} remains true for all $z\in C_\mu$ such that $z\not\in\Sp(\Dvec^\Omega_\mu)$ by analytic continuation. 
\end{proof}

The characterization of the adjoint is again a key ingredient in this section. In contrast with the case of compactly supported perturbations, the presence of the boundary requires a more careful argument.

\begin{lemma}\label{DistortedAdjointMIT}
    With domains $\Dfrak$ one has $(\Dvec^\Omega_\mu)^*=\Dvec^\Omega_{\overline{\mu}}$.
\end{lemma}

\begin{proof}

    Let $\psi\in\Dom((\Dvec^\Omega_\mu)^*)$. Pick $\chi\in C^1_c(\overline{\Omega})$ real valued and supported in $\lcb F=0\rcb$  such that $\indic_{\R^3\setminus\Omega}\prec\chi$ and let us decompose $\psi=\chi\psi+(1-\chi)\psi$. We first work near the boundary. Given $\phi\in\Dfrak$ then $\chi\phi\in\Dfrak$ and 
    \begin{align*}
        \la\chi\phi,(\Dvec^\Omega_\mu)^*\psi\ra_{L^2(\Omega)}&=\la\Dvec^\Omega_\mu(\chi\phi),\psi\ra_{L^2(\Omega)}\\
        &=\la-i\alpha\cdot(\nabla\chi)\phi+\chi\Dvec^\Omega\phi,\psi\ra_{L^2(\Omega)}
    \end{align*}
    hence we get
    \[\la\Dvec^\Omega\phi,\chi\psi\ra_{L^2(\Omega)}=\la\phi,\chi(\Dvec^\Omega_\mu)^*\psi-i\alpha\cdot(\nabla\chi)\psi\ra_{L^2(\Omega)}.\]
    Since $\Dvec^\Omega$ is self-adjoint on $\Dfrak$ then $\chi\psi\in\Dfrak$ and
    \begin{equation}\label{Boundary}
        \Dvec^\Omega(\chi\psi)=\chi(\Dvec^\Omega_\mu)^*\psi-i\alpha\cdot(\nabla\chi)\psi
    \end{equation}
    Now we treat the part away from the boundary. Given $\phi\in H^1$ then $(1-\chi)\phi_{\tq\Omega}\in\Dfrak$ and
    \begin{align*}
        \la(1-\chi)\phi_{\tq\Omega},(\Dvec^\Omega_\mu)^*\psi\ra_{L^2(\Omega)}&=\la\Dvec^\Omega_\mu((1-\chi)\phi_{\tq\Omega}),\psi\ra_{L^2(\Omega)}\\
        &=\la i\alpha\cdot(\nabla\chi)\phi_{\tq\Omega}+(1-\chi)\Dvec^\Omega_\mu\phi_{\tq\Omega},\psi\ra_{L^2(\Omega)}
    \end{align*}
    hence we get
    \[\la\Dvec_\mu\phi,e_\Omega(1-\chi)\psi\ra_{L^2}=\la\phi,e_\Omega((1-\chi)(\Dvec^\Omega_\mu)^*\psi+i\alpha\cdot(\nabla\chi)\psi)\ra_{L^2}.\]
    Thanks to Proposition \ref{FreeDistortedAdjoint} we conclude that $e_\Omega(1-\chi)\psi\in H^1$ and
    \begin{equation}\label{Bulk}
        \Dvec_{\overline{\mu}}(1-\chi)e_\Omega\psi=(1-\chi)e_\Omega(\Dvec^\Omega_\mu)^*\psi+i\alpha\cdot(\nabla\chi)e_\Omega\psi.
    \end{equation}
    Combining \eqref{Boundary} and \eqref{Bulk} we conclude that $\psi\in\Dfrak$ and $(\Dvec^\Omega_\mu)^*\psi=\Dvec^\Omega_{\overline{\mu}}\psi$.

    With the same approach we prove that $\Dfrak\subset\Dom((\Dvec^\Omega_\mu)^*)$ which ends the proof of the Lemma.
\end{proof}

We continue the proof of Theorem \ref{MITLinkResonanceEigenvalue} by the preservation of the multiplicities, exactly as in the previous section. We will only emphasize the few modifications required by the boundary condition. 

\begin{proposition}\label{MITTruncatedRanks}
    Let $\chi\in C^1_c(\overline{\Omega})$ such that $\indic_{\R^3\setminus\Omega}\prec\chi$. 
    \begin{enumerate}
        \item If $\lambda\in\Sp_\disc(\Dvec^\Omega_\mu)$ and $\Pi$ is the associated Riesz projection then $\rank(\Pi)=\rank(\chi\Pi\chi)$.
        \item If $(\lambda,\kappa)\in\Res(\Dvec^\Omega)$ with $\Im(1+\mu)\kappa>0$ and $\Ppaz$ is the residue of $z\mapsto\Rpaz(z,\omega_\mu(z))$ at $\lambda$ then $\rank(\Ppaz)=\rank(\chi\Ppaz\chi)$.
    \end{enumerate}
\end{proposition}

\begin{proof} 
    \textbf{Statement (1).} The map
    \[N:\psi\in\Ran(\Pi)\mapsto(\Dvec_\mu^\Omega-\lambda)\psi\in\Ran(\Pi)\]
    is well defined and nilpotent. Let $\psi\in\Ran(\Pi)$ such that $\chi\psi=0$. Let $m$ be the first integer such that $N^m\psi=0$. Assume that $m>0$. Since $\psi=0$ in a neighborhood of $\partial\Omega$ then $N^{m-1}\psi$ does so. We deduce that $e_\Omega N^{m-1}\psi\in H^1$ is an eigenvector of $\Dvec_\mu$ which is a contradiction by Proposition \ref{DistortedResolvent}. Hence  $\psi=0$ i.e the multiplication by $\chi$ is one-to-one on $\Ran(\Pi)$. Similarly the multiplication by $\overline{\chi}$ is one-to-one on $\Ran(\Pi^*)$ thanks to Lemma \ref{DistortedAdjointMIT}. We conclude as for the proof of Proposition \ref{TruncatedRanks}.
    
    \textbf{Statement (2).} The endomorphism
    \[\Npaz :\psi\in\Ran(\Ppaz)\mapsto(\beta-i\alpha\cdot\nabla-\lambda)\psi\in\Ran(\Ppaz)\]
    is well defined and nilpotent. Let $\psi\in\Ran(\Ppaz)$ such that $\chi\psi=0$. Assume that the first integer $m$ such that $\Npaz^m\psi=0$ is positive. Then $\Npaz^{m-1}\psi=0$ in a neighborhood of $\partial\Omega$ and $(\beta-i\alpha\cdot\nabla-\lambda)\Npaz^{m-1}\psi=0$. In particular $\Npaz^{m-1}\psi$ extends as a function $\Psi\in H^2_\loc$ such that $(-\Delta-\kappa^2)\Psi=0$. We conclude that $\Psi=0$ by unique continuation \cite[Theorem XIII.63]{ReSi72}. Eventually $\psi=0$ i.e the multiplication by $\chi$ is one-to-one on $\Ran(\Ppaz)$.
    
    We recall that if $I\geq 1$ is the index of $\Npaz$ then
    \[z\mapsto \Rpaz(z,\omega_\mu(z))-\sum_{i=1}^I\frac{\Npaz^{i-1}\Ppaz}{(z-\lambda)^i}\]
    has a holomorphic continuation from an open neighbourhood of $\lambda$ to $\Lcal(L^2_\comp(\overline{\Omega}),\Dfrak_\loc)$. 
    Thanks to \eqref{RightParametrix} with $\mu=0$ and the analytic continuation, one has for $(z,\omega)\in\Mpaz$ which is not resonance of $\Dvec^\Omega$
    \begin{equation*}
        \Rpaz(z,\omega)=\Tpaz(z,\omega)-\Rpaz(z,\omega)\Kpaz(z,\omega)
    \end{equation*}
    where we set, with $\Rpaz_0$ as in the proof of Theorem \ref{TruncatedRanks},
    \begin{equation*}
        \left\{
        \begin{array}{ll}
            \Tpaz(z,\omega):=r_\Omega(1-\rho_0)\Rpaz_0(z,\omega)e_\Omega+\rho_{1\tq\Omega}(\Dvec^\Omega-z_0)^{-1}\rho_{0\tq\Omega}\\
            
            \Kpaz(z,\omega):=ir_\Omega\alpha\cdot(\nabla\rho_0)\Rpaz_0(z,\omega)e_\Omega+\lp i\alpha\cdot(\nabla\rho_{1\tq\Omega})-(z-z_0)\rho_{1\tq\Omega}\rp(\Dvec^\Omega-z_0)^{-1}\rho_{0\tq\Omega}.
        \end{array}
        \right.
    \end{equation*}
    We deduce
    \begin{equation}\label{MITResiduesComparison}
        \Ppaz=-\sum_{i=0}^{I-1}\Npaz ^i\Ppaz\dfrac{d^i}{dz^i}_{\tq z=\lambda}\Kpaz(z,\omega_\mu(z)).
    \end{equation}
    Notice that we can choose $\indic_{\R^3\setminus\Omega}\prec\rho_0\prec\rho_1\prec\chi$. We construct $\chi_1,...,\chi_{I-1}\in C^1_c(\overline{\Omega})$ such that $\rho_1\prec\chi_{I-1}\prec...\prec\chi_1\prec\chi$ and
    \[\Ran(\Npaz ^{I-1}\Ppaz\chi_{I-1})\subset...\subset\Ran(\Npaz\Ppaz\chi_1)\subset\Ran(\Ppaz\chi).\]
    Since we can replace $\Npaz ^i\Ppaz$ by $\Npaz ^i\Ppaz\chi_i$ in \eqref{MITResiduesComparison} we deduce $\Ran(\Ppaz)=\Ran(\Ppaz\chi)$ and we easily conclude.
\end{proof}

Finally, the proof of Theorem \ref{MITLinkResonanceEigenvalue} is exactly the same as the proof of Theorem \ref{LinkResonanceEigenvalue}. Again we can apply Theorem \ref{MITLinkResonanceEigenvalue} to recover Kramers degeneracy for resonances.

\begin{corollary}\label{MITKramers}
    Every resonance $(\lambda,\kappa)\in\Mpaz$ of $\Dvec^\Omega$ such that $\kappa\not\in i(-\infty,0]$ has even multiplicity. 
\end{corollary}

\begin{proof}
    Thanks to Lemma \ref{FChoice} (with $K=\R^3\setminus\Omega$) and Lemma \ref{CatchedResonances} there is a suitable choice of a deformation $F$ and a distortion parameter $\mu$ such that $\lambda\in\Sp(\Dvec^\Omega_\mu)$ and the multiplicity of $(\lambda,\kappa)$ is the rank of the associated Riesz projection. Again we consider $T\psi:=-i\gamma_5\alpha_3\overline{\psi}$ for $\psi\in L^2(\Omega)$. Since $T\beta=\beta T$ and $T\alpha_j=-\alpha_jT$ then  $T\Dfrak\subset\Dfrak$ and $T\Dvec^\Omega_\mu=\Dvec^\Omega_{\overline{\mu}}T$. Thanks to Lemma \ref{DistortedAdjointMIT} we can mimic the proof of Corollary \ref{KramersDegeneracy} to conclude. 
\end{proof}

\appendix 

\section{Technical proofs}\label{Appendix}

\subsection{Proof of Proposition \ref{DistortedDerivatives}}

\begin{proof}
    Recall that $f_\mu(x)=x+\mu F(x)$. Since $\lv\mu\rv<1/L$ then
    \begin{align*}
        \det\nabla f_\mu(x)&\geq\lp1-\lV\mu \nabla F(x)\rV\rp^3\\
        &>0
    \end{align*}
    therefore $\nabla f_\mu(x)$ is invertible. Since $ x\mapsto\mu F(x)$ is a contraction mapping then $ f_\mu$ is one-to-one. Given $y\in\R^3$ the map $ x\mapsto y-\mu F(x)$ has a fixed point thanks to Banach-Picard Theorem therefore $ f_\mu(\R^3)=\R^3$. We conclude that $ f_\mu$ is a $ C^1-$diffeomorphism $\R^3$ to $\R^3$ thanks to the global inverse Theorem. We deduce that $U_\mu$ is an isometry of $ L^2$ thanks to the change of variables formula. Secondly, $U_\mu$ is invertible and 
    \[(U_\mu^{-1}\psi)(x):=\sqrt{\det \nabla f_\mu^{-1}(x)}\psi(f_\mu^{-1}(x)).\]
    thanks to the chain rule. This proves the unitarity property.
    
    Thanks to the chain rule, Leibniz formula and the differential of the determinant we find
    \begin{align*}
        \lp\partial_{x_j} U_\mu^{-1}\psi\rp(x)
        &=\frac{1}{2}\sqrt{\det \nabla f_\mu^{-1}(x)}\tr\lp \nabla f_\mu^{-1}(x)^{-1}\partial_{x_j}\nabla f_\mu^{-1}(x)\rp\psi(f_\mu^{-1}(x))\\
        &+\sqrt{\det \nabla f_\mu^{-1}(x)}\sum_{k=1}^3[\nabla f_\mu^{-1}(x)]_{kj}\partial_{x_k}\psi(f_\mu^{-1}(x))\\
        &=\frac{1}{2}\tr\lp \nabla f_\mu^{-1}(x)^{-1}\partial_{x_j}\nabla f_\mu^{-1}(x)\rp \lp U_\mu^{-1}\psi\rp(x)+\sum_{k=1}^3[\nabla f_\mu^{-1}(x)]_{kj}\lp U_\mu^{-1}\partial_{x_k}\psi\rp(x).
    \end{align*}
    In particular if $\psi\in H^1$ then $U_\mu^{-1}\psi\in H^1$. Also, since $f_\mu(f_\mu^{-1}(x))= x$ then 
    \begin{equation*}
        \nabla f_\mu(f_\mu^{-1}(x))\nabla f_\mu^{-1}(x)=I_3
    \end{equation*} and
    \[\sum_{k=1}^3[\nabla f_\mu^{-1}(x)]_{kj}\partial_{x_k}\nabla f_\mu(f_\mu^{-1}(x))\nabla f_\mu^{-1}(x)+\nabla f_\mu(f_\mu^{-1}(x))\partial_{x_j}\nabla f_\mu^{-1}(x)=0_3\]
    which provides the expected conclusion. 
    
\end{proof}

\subsection{Proof of Lemma \ref{TAnalyticity}}

\begin{proof}
    We need estimates on $K_\mu^z$ and take care of the dependance in $z$ and $\mu$ to ensure the analyticity properties.

    Clearly
    \begin{equation}\label{JacobianEstimates}
        \lv \rho_\mu(x)\rv\lesssim1.
    \end{equation}
    We continue with a lower bound on $\delta_\mu$. Factorizing $\lv x-y\rv$ we get
    \begin{equation}\label{NBelowEstimate}
        \lv \delta_\mu(x,y)\rv\geq c_\mu\lv x-y\rv
    \end{equation}
    where 
    \[c_\mu:=\inf_{\lv p\rv=1,\lv q\rv\leq L}\lv\sqrt{\lp p+\mu q\rp^2}\rv>0.\]
    As in Remark \ref{SquareRootDefinition} the square root in the infimimum makes sense. Also
    \begin{align*}
        &\delta_\mu(x,y)-(1+\mu)\lv x-y\rv\\
        &=\frac{
        (f_\mu(x)-f_\mu(y)-(1+\mu)(x-y))^\Tsf(f_\mu(x)-f_\mu(y)+(1+\mu)(x-y))}{\delta_\mu(x,y)+(1+\mu)\lv x-y\rv}\\
        &=\frac{\mu
        (F(x)-x+y-F(y))^\Tsf(F(x)-F(y)+(1+2\mu)(x-y))}{\delta_\mu(x,y)+(1+\mu)\lv x-y\rv}.
    \end{align*}
    Since $\Re\delta_\mu(x,y)>0$ then
    \[\lv \delta_\mu(x,y)+(1+\mu)\lv x-y\rv\rv>(1+\Re\mu)\lv x-y\rv\]
    and since $F(x)=x$ for $\lv x\rv\gg1$ and $\lv F(x)-F(y)\rv\lesssim\lv x-y\rv$ then
    \[\lv\mu(F(x)-x+y-F(y))^\Tsf(F(x)-F(y)+(1+2\mu)(x-y))\rv\lesssim\lv x-y\rv.\]
    Consequently
    \begin{equation}\label{NBehaviour}
        \lv \delta_\mu(x,y)-(1+\mu)\lv x-y\rv\rv\leq\frac{C}{1+\Re\mu}
    \end{equation}
    for some constant $C$ independent of $ x$, $y$ and $\mu$. Combining  \eqref{JacobianEstimates}, \eqref{NBelowEstimate} and \eqref{NBehaviour} we get
    \begin{equation}
        \lV K_\mu^z(x,y)\rV\lesssim\frac{e^{C\frac{\lv\omega_\mu(z)\rv}{1+\Re\mu}}}{c_\mu}\frac{e^{-\lv x-y\rv\Im\lp(1+\mu)\omega_\mu(z)\rp}}{\lv x-y\rv}\lp1+\lv z\rv+\frac{\lv\omega_\mu(z)\rv}{c_\mu}+\frac{1}{c_\mu^2}\frac{1}{\lv x-y\rv}\rp.
    \end{equation}
    
    As a consequence
    \[\sup_{ x\in\R^3}\int_{\R^3}\lV K_\mu^z(x,y)\rV dy, \sup_{y\in\R^3}\int_{\R^3}\lV K_\mu^z(x,y)\rV d x\lesssim\int_{\R^3}\frac{e^{-\lv v\rv\Im\lp(1+\mu)\omega_\mu(z)\rp}}{\lv v\rv}\lp1+\frac{1}{\lv v\rv}\rp dv\]
    and the integral in the r.h.s. is finite since $\Im(1+\mu)\omega_\mu(z)>0$. 
    Thanks to Schur's test $T_\mu^z$ is well defined and bounded from $ L^2$ into itself. The analyticity properties are obtained by applying the dominated convergence Theorem to $\la\phi,T_\mu^z\psi\ra_{L^2}$ for any $\psi,\phi\in L^2$.
\end{proof}

\printbibliography

\end{document}